\documentclass[reqno]{amsart}

\usepackage{amsmath,amssymb,amsthm,mathtools,enumitem,setspace}
\usepackage{hyperref}
\usepackage{needspace}
\AddToHook{env/dmlconjecture/before}{\Needspace{8\baselineskip}}
\AddToHook{env/theorem/before}{\Needspace{8\baselineskip}}
\AddToHook{env/lemma/before}{\Needspace{8\baselineskip}}
\AddToHook{env/proposition/before}{\Needspace{8\baselineskip}}
\AddToHook{env/corollary/before}{\Needspace{8\baselineskip}}
\AddToHook{env/example/before}{\Needspace{8\baselineskip}}
\AddToHook{env/definition/before}{\Needspace{6\baselineskip}}

\hypersetup{
  colorlinks=true,
  linkcolor=blue,
  citecolor=blue,
  urlcolor=blue,
  linktoc=all,
  bookmarksopen=true,
  bookmarksnumbered=true
}

\numberwithin{equation}{section}
\newtheorem{theorem}{Theorem}[section]
\newtheorem{lemma}[theorem]{Lemma}
\newtheorem{proposition}[theorem]{Proposition}
\newtheorem{corollary}[theorem]{Corollary}
\newtheorem{dmlconjecture}[theorem]{Conjecture}
\theoremstyle{definition}
\newtheorem{definition}[theorem]{Definition}
\newtheorem{example}[theorem]{Example}
\theoremstyle{remark}
\newtheorem{remark}[theorem]{Remark}
\usepackage[left=1.3in, right=1.3in, top=1.1in, bottom=1.1in]{geometry}

\newcommand{\Nzero}{\mathbb N_0}
\newcommand{\A}{\mathbb A}
\newcommand{\Pj}{\mathbb P}
\newcommand{\cI}{\mathcal I}
\newcommand{\cO}{\mathcal O}
\newcommand{\cK}{\mathcal K}
\newcommand{\cF}{\mathcal F}
\newcommand{\comp}{\operatorname{comp}}

\begin{document}

\title[Ramified Skew Products over $\mathbb C$]{Dynamical Mordell--Lang Conjecture for Higher-Rank Radially Ramified Skew Products}

\author{Bhawesh Mishra}
\email{bhaweshmishra2024@gmail.com}
\address{Previously at: 384 Dunn Hall, Department of Mathematical Sciences, The University of Memphis, Memphis, TN 38157}
\curraddr{Ramjanki Path, Biratnagar, Morang, Kosi, Nepal 56613}

\subjclass[2020]{Primary 37P55, 37P05; Secondary 11B37, 11S82}
\keywords{dynamical Mordell--Lang; ramified polynomial skew products;
higher-rank fibers; critical absorption; blow-ups; $p$-adic interpolation}
\date{}

\begin{abstract}
We establish the dynamical Mordell--Lang conjecture over the complex
numbers for a family of radially ramified polynomial skew products
in arbitrary base dimension and fiber rank. We assume that one fixed
iterate sends the vertical critical locus into the invariant zero
section. Our result covers affine-linear bases and, under a degree-gap
condition, nonlinear \'{e}tale polynomial bases.
We also establish a uniform Skolem--Mahler--Lech theorem for a fixed
nonlinear orbit and a fixed polynomial-coefficient linear recurrence
whose trailing coefficient is a nonzero constant. The theorem applies
to joint polynomial and rational relations involving the orbit and
finitely many successive terms of any solution. A single eventual
period works for every such solution and relation, although the finite
exceptional set may depend on the relation. Our proof combines
controlled nonarchimedean realizations with an
orbitwise transfer theorem for superattracting line-bundle dynamics.
\end{abstract}

\maketitle

\section{Introduction}

The dynamical Mordell--Lang conjecture concerns the intersection of a
forward orbit under an endomorphism with a closed subvariety. For an
endomorphism $f:X\to X$ of a quasi-projective variety, a point $x\in X$,
and a closed subvariety $V\subseteq X$, it predicts that the indices $n$
for which $f^n(x)\in V$ form a finite union of arithmetic progressions.
We prove this conclusion for a family of higher-rank radially ramified
skew products. This arithmetic-progression structure extends the
pattern familiar from the Skolem--Mahler--Lech theorem and from
cyclic-subgroup intersections in classical Mordell--Lang
\cite{Lech1953,Faltings1994,Vojta1996,DMLBook}.

\begin{definition}[Iterates and return sets]\label{def:return-set}
Let $f:X\to X$ be an endomorphism of a quasi-projective variety, let
$x$ be a point of $X$, and let $V\subseteq X$ be closed. Write
$\Nzero=\{0,1,2,\ldots\}$, let $f^n$ denote the $n$th iterate of $f$
(with $f^0=\mathrm{id}_X$, the identity map of $X$), and define the
\emph{return set}
$S_f(x,V)=\{n\in\Nzero:f^n(x)\in V\}$.
\end{definition}

We use the following formulation of the conjecture, due to Ghioca and
Tucker \cite{Ghioca2009}.

\begin{dmlconjecture}[Dynamical Mordell--Lang]\label{conj:dml}
Let $K$ be a field of characteristic zero, let $X$ be a
quasi-projective $K$-variety, let $f:X\to X$ be a $K$-endomorphism,
let $x\in X(K)$, and let $V\subseteq X$ be closed. Then $S_f(x,V)$ is
a finite union of sets
$a+b\Nzero=\{a+bn:n\in\Nzero\}$, with $a,b\in\Nzero$.
The case $b=0$ is allowed and represents a singleton.
\end{dmlconjecture}

Thus a forward orbit meets a closed subvariety along finitely many
entire arithmetic subsequences and at only finitely many additional
times. The orbit itself need not be periodic. The conclusion
distinguishes persistent algebraic intersections from isolated
coincidences rather than merely bounding the frequency of visits
\cite{DMLBook}.

Our main theorem establishes the conjecture for a nonlinear ramified
family in arbitrary base dimension and fiber rank. The radial form and
the critical-absorption condition are defined in
Definitions~\ref{def:radial-skew-product}
and~\ref{def:critical-absorption}.

\begin{theorem}\label{thm:main}
Let $k$ be an algebraically closed field of characteristic zero, and
let $m\geq0$ and $r\geq1$. Let
$T:\A^m_k\to\A^m_k$ be an affine-linear endomorphism, and choose
$N\geq0$ such that
$S=T^N(\A^m)=T^{N+1}(\A^m)$.
Let
$M(u)\in\operatorname{GL}_r(k[u_1,\ldots,u_m])$, let $d\geq2$, and
let
\[
 0\ne\rho(u,\boldsymbol z)\in
 (z_1,\ldots,z_r)^{d-1}
 k[u_1,\ldots,u_m,z_1,\ldots,z_r].
\]
Define
$F(u,\boldsymbol z)=
\bigl(T(u),\rho(u,\boldsymbol z)M(u)\boldsymbol z\bigr)$.

On $S\times\A^r$, put $Z=S\times\{0\}$ and
$C=V\!\left(
\det D_{\boldsymbol z}
\bigl(\rho(u,\boldsymbol z)M(u)\boldsymbol z\bigr)
\right)$.
If, for some $\ell\geq1$, one has the set-theoretic inclusion
$\bigl(F|_{S\times\A^r}\bigr)^\ell(C)\subseteq Z$,
then $F$ satisfies the Dynamical Mordell--Lang conjecture.
Equivalently, for every $x\in\A^{m+r}(k)$ and every closed subvariety
$V\subseteq\A^{m+r}_k$, the set
$\{n\in\Nzero:F^n(x)\in V\}$ is a finite union of arithmetic
progressions.
\end{theorem}

We illustrate this family with the concrete instance
$F(u,x,y)=\bigl(u+1,(x+uy)^2,(x+uy)y\bigr)$ on $\A^3$, with base
coordinate $u$ and fiber coordinates $x,y$.
Example~\ref{ex:matrix-cocycle} specifies its matrix and radial factor
and verifies critical absorption. The two fiber coordinates are
coupled through both the matrix and the common factor.

The principal distinction between our setting and the higher-rank
skew-linear results of Ghioca--Xie lies in ramification rather than
fiber dimension. Their results already allow arbitrary fiber
dimension over an algebraically closed field of characteristic zero
\cite{GhiocaXieSkewLinear}. In the present family, however, the common
scalar factor depends on the fiber variables, and the fiber
differential vanishes along the zero section. The general \'{e}tale
theorem \cite{Bell2010} is therefore inapplicable. Critical absorption
controls the entire vertical critical locus without removing this
ramification. Theorem~\ref{thm:main} gives the full return-set
conclusion for every closed target in the stated family and is
complementary to, rather than a formal strengthening of, the existing
theorems.

Our proof uses two principal ingredients. First, we blow up the zero
section to separate the projective direction from the radial motion.
Second, we use Proposition~\ref{prop:controlled-place} to construct
nonarchimedean embeddings for which the valuations of varying
algebraic expressions are bounded by their presentation complexity.
In Theorem~\ref{thm:line-bundle}, we compare these bounds with the
superattracting radial contraction. That theorem is formulated for
line-bundle dynamics over an \'{e}tale base map and imposes its
presentation-growth hypothesis only on the chosen base orbit.

We obtain a particularly explicit subfamily when $\rho$ is
homogeneous of fiber degree $d-1$. In that case, the vertical critical
locus is $V(\rho)$ and maps into $Z$ in one step. Consequently,
Corollary~\ref{cor:homogeneous} applies without a separate
critical-absorption hypothesis. For nonlinear \'{e}tale bases, the
degree-gap hypothesis of Corollary~\ref{cor:degree-gap} remains in
force.

\subsection{Notation and terminology}
\label{subsec:terminology}

\begin{definition}[Ambient spaces and algebraic notation]\label{def:ambient-notation}
We write $\A^m_K$, $\Pj^m_K$, and $X(K)$ for affine space, projective
space, and the set of $K$-points, respectively, and omit the ground
field when it is understood. All closed sets and closures are Zariski.
For a commutative ring $B$, write $B^\times$ for its group of units,
$M_N(B)$ for the ring of $N$-by-$N$ matrices,
$\operatorname{GL}_N(B)$ for its group of invertible matrices, and
$I_N$ for the identity matrix. Polynomial and rational function rings
are denoted by $B[x_1,\ldots,x_t]$ and $K(x_1,\ldots,x_t)$.
\end{definition}

\begin{definition}[Arithmetic progressions and periodicity]\label{def:periodicity}
Arithmetic progressions include the singleton case $a+0\Nzero$.
A subset of $\Nzero$ is \emph{eventually periodic} if, beyond some
index, membership depends only on the residue class modulo a fixed
positive integer. An \emph{orbit tail} is obtained by deleting
finitely many initial iterates. A point or subvariety $Y$ is
\emph{periodic} if $f^n(Y)=Y$ for some $n\geq1$, and a subset is
\emph{invariant} if the map sends it into itself. We write
$f|_U$ for restriction and $f^{-j}(W)$ for the inverse image of $W$
under $f^j$, without assuming that $f$ is invertible.
\end{definition}

\subsection{The radial family and consequences}
\label{subsec:principal-results}

\begin{definition}[Radial skew products]\label{def:radial-skew-product}
Fix integers $m\geq0$ and $r\geq1$, and write
$u=(u_1,\ldots,u_m)$ and $\boldsymbol z=(z_1,\ldots,z_r)$ for the
base and fiber coordinates. Let $T:\A^m\to\A^m$ be affine-linear,
let $M(u)$ be a polynomially invertible $r$-by-$r$ matrix, and fix
$d\geq2$. Polynomial invertibility means that both $M(u)$ and
$M(u)^{-1}$ have polynomial entries, equivalently that $\det M(u)$ is
a nonzero constant. Assume that
$\rho(u,\boldsymbol z)\in(z_1,\ldots,z_r)^{d-1}$.
The \emph{fiber degree} of a monomial is its total degree in the
variables $z_1,\ldots,z_r$; thus the $q$th power of
$(z_1,\ldots,z_r)$ consists precisely of polynomials all of whose
monomials have fiber degree at least $q$. Define
\begin{equation}\label{eq:intro-map}
 F(u,\boldsymbol z)=
 \bigl(T(u),\rho(u,\boldsymbol z)M(u)\boldsymbol z\bigr).
\end{equation}
The map \eqref{eq:intro-map} is a \emph{skew product}: its base
coordinate depends only on $u$. Its \emph{zero section}
$\boldsymbol z=0$ is invariant, and its fiber map vanishes there to
order at least $d$. We call $\rho$ the \emph{radial factor}, $d$ the
lower bound for the \emph{radial order}, and
$\boldsymbol z\mapsto M(u)\boldsymbol z$ the underlying linear
cocycle. Because $\rho$ may depend on $\boldsymbol z$,
\eqref{eq:intro-map} need not be skew-linear.
\end{definition}

For an affine-linear map $T$, the descending chain
$T^n(\A^m)$ stabilizes because every strict inclusion lowers the
dimension. If $S=T^N(\A^m)=T^{N+1}(\A^m)$, then $T|_S$ is a
surjective affine-linear self-map of $S$ and hence an affine
automorphism.

\begin{definition}[Vertical critical locus and critical absorption]\label{def:critical-absorption}
On $S\times\A^r$, let $Z=S\times\{0\}$. The \emph{vertical critical
locus} $C$ is the zero set of the determinant of the Jacobian of the
fiber map with respect to $\boldsymbol z$, with $u$ fixed. The
condition $\bigl(F|_{S\times\A^r}\bigr)^\ell(C)\subseteq Z$
for some $\ell\geq1$ is called \emph{critical absorption}; the
inclusion is set-theoretic. For polynomials $h_1,\ldots,h_t$, write
$V(h_1,\ldots,h_t)$ for their common zero set, and write
$D_{\boldsymbol z}$ for the Jacobian with respect to the fiber
coordinates.
\end{definition}

We write $E_{\boldsymbol z}=\sum_i z_i\partial/\partial z_i$ for the
fiber Euler operator. The vertical critical determinant is
\begin{equation}\label{eq:general-critical-determinant}
 \det D_{\boldsymbol z}
 \bigl(\rho(u,\boldsymbol z)M(u)\boldsymbol z\bigr)
 =\det(M(u))\rho(u,\boldsymbol z)^{r-1}
   \bigl(\rho(u,\boldsymbol z)+E_{\boldsymbol z}\rho(u,\boldsymbol z)\bigr).
\end{equation}
Indeed,
\[
D_{\boldsymbol z}(\rho M\boldsymbol z)
=M\bigl(\rho I_r+\boldsymbol z(\nabla_{\boldsymbol z}\rho)^{\mathsf T}\bigr),
\]
and the rank-one determinant identity
\begin{equation}\label{eq:rank-one-determinant}
 \det(aI+vw^{\mathsf T})=a^{r-1}(a+w^{\mathsf T}v).
\end{equation}
Together with the preceding Jacobian identity, this proves
\eqref{eq:general-critical-determinant}.
Thus critical absorption concerns the zero set of the polynomial in
\eqref{eq:general-critical-determinant}, not a scheme-theoretic
containment. The specialization $r=1$, $M=1$, and
$\rho=z^{d-1}g$ gives the one-dimensional fiber map
$z\mapsto z^d g(u,z)$. In higher rank, both $M(u)$ and the common
factor may couple the fiber coordinates.

For a self-map of a smooth variety, \'{e}tale means that its
differential is everywhere invertible; over general base rings we use
the scheme-theoretic convention stated in
Section~\ref{subsec:geometric-inputs}. All polynomial degrees are total
degrees, and the degree of a polynomial map is the maximum of the
degrees of its nonzero coordinate polynomials.

\begin{definition}[Asymptotic notation]\label{def:asymptotic-growth}
For a positive sequence $b_n$, the relation $a_n=O(b_n)$ means that
$|a_n|\leq Cb_n$ for all sufficiently large $n$ and some constant
$C$, while $a_n=o(b_n)$ means that $a_n/b_n\to0$. Unless stated
otherwise, all asymptotic bounds are taken as $n\to\infty$ with the
remaining data fixed. This notation is distinct from the valuation
ring $O$ used below.
\end{definition}

In Corollary~\ref{cor:degree-gap}, we extend the return-set conclusion to
an \'{e}tale polynomial base $T$ under critical absorption and the
additional inequality $\deg(T^q)<d^q$ for some $q\geq1$. After passage
to the $q$th iterate, we use Lemma~\ref{lem:cocycle} to bound the
projective-direction presentation growth by a quantity of smaller
order than the radial scale $d^{qn}$. Example~\ref{ex:nonlinear-base}
provides an explicit nonlinear base with a nonconstant matrix cocycle.

If $P(u,\boldsymbol z)$ is a nonzero polynomial homogeneous of fiber
degree $d-1$, we combine Euler's identity
$E_{\boldsymbol z}P=(d-1)P$ and
\eqref{eq:rank-one-determinant} to obtain
\begin{equation}\label{eq:intro-determinant}
 \det D_{\boldsymbol z}
 \bigl(P(u,\boldsymbol z)M(u)\boldsymbol z\bigr)
 =d\det(M(u))P(u,\boldsymbol z)^r.
\end{equation}
Consequently, the critical locus is $V(P)$ and its image is contained
in $Z$. Corollary~\ref{cor:homogeneous} therefore requires no separate
critical-absorption hypothesis. For a nonlinear \'{e}tale base, the
degree-gap condition remains in force. Example~\ref{ex:nonhomogeneous-radial}
shows that critical absorption may also hold for a nonhomogeneous
factor and may require more than one iterate.

In Theorem~\ref{thm:uniform-sml}, we prove a uniform
Skolem--Mahler--Lech consequence for a fixed nonlinear orbit and a
polynomial-coefficient linear recurrence of the form
\eqref{eq:scalar-recurrence}, with nonzero constant coefficient on its
earliest term. A single eventual period applies to all permitted
polynomial and rational relations, to every finite number of
successive terms, and to all solutions of the fixed recurrence; the
finite exceptional set may depend on the relation. Rational
expressions are considered only at indices where their denominators do
not vanish. The common-period argument uses the sequence-ring method
of Wibmer \cite[proof of Theorem~3.1]{Wibmer2015} and
\cite[Appendix, proof of Theorem~4.2]{GhiocaXieSkewLinear}.

\subsection{Arithmetic origins and geometric consequences}
\label{subsec:significance}

The classical prototypes for our setting are the Skolem--Mahler--Lech theorem
\cite{Lech1953} and the Mordell--Lang theorem for semiabelian varieties
\cite{Faltings1994,Vojta1996}. Consecutive terms of a
constant-coefficient linear recurrence evolve under a companion
linear map, so the zero set of the recurrence is the return set to a
coordinate hyperplane \cite[Section~2.2]{XieAroundDML}. Similarly,
nonnegative multiples of a point form a translation orbit, and
intersections with a cyclic subgroup yield arithmetic progressions
\cite[Section~1]{Bell2010}. Dynamical Mordell--Lang extends this
one-parameter intersection principle beyond linear recurrences and
group varieties. Denis sought arithmetic progressions among recurring
intersections \cite{Denis1994}; the full conjecture requires such
progressions to account for every return apart from finitely many
exceptions.

In characteristic zero, the geometric form of the conjecture predicts
that an irreducible positive-dimensional subvariety meeting a forward
orbit in a Zariski-dense set is periodic
\cite[Conjecture~1.2]{XieAroundDML}. Exact return-set theorems therefore
convert repeated incidences into invariant algebraic structure. Bell,
Ghioca, and Tucker proved that an infinite Zariski-dense orbit of an
automorphism of an irreducible quasi-projective variety in
characteristic zero is \emph{critically dense}: every infinite subset
is Zariski dense \cite[Section~5]{Bell2010}. Keeler, Rogalski, and
Stafford used critical density to prove the Noetherian property of
their noncommutative blowup algebras under the remaining hypotheses of
their construction
\cite[Theorem~1.1]{KeelerRogalskiStafford2005}. These applications,
together with the difference-equation consequence in
Theorem~\ref{thm:uniform-sml}, require a finite exceptional set rather
than a mere sparsity estimate.

\subsection{Known cases and recent methods}
\label{subsec:related-work}

Bell, Ghioca, and Tucker proved Dynamical Mordell--Lang for every point
and every closed target under an \'{e}tale endomorphism of a
quasi-projective variety over $\mathbb C$
\cite[Theorem~1.3]{Bell2010}, extending Bell's theorem for affine
variety automorphisms \cite{Bell2006Affine,Bell2008Corrigendum}. Their
argument used $p$-adic interpolation of orbit subsequences and the
finiteness of analytic zero sets \cite[Section~1]{Bell2010}. This
result, recalled as Theorem~\ref{thm:etale-dml}, applies to the present
dynamics on the zero section but not to the ramified fiber dynamics.

Xie proved the conjecture for all polynomial endomorphisms of
$\A^2_{\overline{\mathbb Q}}$ \cite{Xie17} and subsequently extended
the result to $\mathbb C$ \cite[Theorem~3.2]{XieAroundDML}. The
extension used Moriwaki's height theory over finitely generated fields
\cite{Moriwaki2000} and the finitely generated-ring form of Siegel's
theorem \cite[Section~3.1]{XieAroundDML}. We use these results as points
of comparison rather than as inputs to the proofs below. In particular,
the return-set conclusion of the present paper was already known in
total dimension two, whereas the family considered here permits
arbitrary base dimension and fiber rank.

Over an algebraically closed field $k$ of characteristic zero,
Ghioca--Xie considered skew-linear maps
$(x,y)\mapsto\bigl(g(x),A(x)y\bigr)$ on $\A^1\times\A^N$. When $g$
was an automorphism and $A(x)\in M_N(k[x])$, they proved that an
irreducible curve containing infinitely many distinct points of an
orbit was periodic \cite[Theorem~1.1]{GhiocaXieSkewLinear}. They also
treated rational maps on $\Pj^1\times\A^N$ with
$g:\Pj^1\to\Pj^1$ of degree greater than one and
$A(x)\in M_N(k(x))$: every irreducible positive-dimensional
subvariety whose intersection with a well-defined forward orbit was
Zariski dense was periodic
\cite[Theorem~1.2]{GhiocaXieSkewLinear}. For this rational-map result,
images of subvarieties were taken to be the Zariski closures of the
images on the domain of definition
\cite[Section~1.1]{GhiocaXieSkewLinear}. These theorems already allowed
arbitrary fiber dimension. The distinction in
\eqref{eq:intro-map} is the nonlinear scalar factor depending on the
fiber variables.

Without restrictions on ramification, Bell, Ghioca, and Tucker
obtained a weak theorem for rational self-maps of quasi-projective
varieties over an arbitrary field and for well-defined forward orbits:
a return set was a finite union of arithmetic progressions together
with a set of Banach density zero
\cite[Corollary~1.5]{BGTNoetherian2015}. The density-zero set could
still be infinite. Xie's uniform weak theorem made the sparsity bound
independent of the initial point among points with Zariski-dense
orbits, for a fixed map and a proper closed target
\cite[Theorem~5.2]{XieAroundDML}. Uniform weak estimates for finite
orbit segments also yielded upper bounds for arithmetic height growth
in terms of degree growth \cite[Theorem~6.1]{XieAroundDML}. By
contrast, we exclude an infinite exceptional set for the stated
ramified family.

In 2026, Xie, Yang, and Zheng proved the full conjecture for split
products $f\times g$ on $X\times Y$, where $X$ was an affine curve and
$Y$ a projective curve over $\overline{\mathbb Q}$
\cite[Theorem~1.1]{XieYangZheng2026SplitCurves}. Other coordinatewise
and split cases were treated in \cite{GTZ08,BGKT12}, while
Medvedev--Scanlon developed the classification of invariant varieties
for such polynomial systems \cite{MedvedevScanlon2014}.

In a 2026 preprint, Yang and Zheng treated regular polynomial
endomorphisms $f$ of $\A^N_{\mathbb C}$, namely those extending to
endomorphisms of $\Pj^N$. They wrote $d=\deg f$, imposed an integer
degree-gap parameter $k$ with $1\leq k\leq d$, and assumed that the
lower homogeneous terms had degree at most $d-k$. If $f_\infty$
denoted the induced map on the hyperplane at infinity, their
multiplicity condition was
\begin{equation}\label{eq:yang-zheng-condition}
 k>2\max_{P\text{ periodic for }f_\infty}e_{f_\infty}(P),
\end{equation}
as stated in \cite[Condition~1.5]{YangZheng2026LocalHeight}. Here
$e_{f_\infty}(P)$ was the local multiplicity, equivalently the length
of the local ring of the scheme-theoretic fiber at $P$
\cite[Section~1.1]{YangZheng2026LocalHeight}. Under
\eqref{eq:yang-zheng-condition}, they proved Dynamical Mordell--Lang for
curve targets and showed that every irreducible $f$-periodic curve was
a line through the origin, called a \emph{vertical line} in their
terminology
\cite[Setting~1.2 and Theorem~1.7]{YangZheng2026LocalHeight}. Their
degree-gap condition differed from the inequality $\deg(T^q)<d^q$ in
Corollary~\ref{cor:degree-gap}.

Growth separation also appeared in the height-theoretic transfer
theorem of Xie--Yang. They proved that, for a surjective endomorphism
of a projective variety and a surjective endomorphism of a projective
curve in characteristic zero, the return sets for the product had the
same form as those for the first factor when the degree of the curve map
exceeded the first dynamical degree of the endomorphism of the
projective variety
\cite[Theorem~1.8]{XieYangHeight2025}. In that theorem, the first
dynamical degree measured the exponential growth of divisor pullbacks.
Their theorem also treated specified classes of return sets in positive
characteristic, without asserting the characteristic-zero conclusion
there. Related comparisons of growth and rates of approach appeared in
\cite[Section~1]{XieYangZheng2026SplitCurves} and
\cite[Section~1.2]{YangZheng2026LocalHeight}. Our comparison is instead
between projective-direction presentation complexity and radial
contraction.

\subsection{Method and structure of the proof}

We begin the proof by blowing up $S\times\A^r$ along $Z$. By
Proposition~\ref{prop:bundle-blowup-input} and
\eqref{eq:blowup-bundle}, we identify this blow-up with the total space of the
tautological line bundle over $S\times\Pj^{r-1}$. The common radial
factor makes the lifted map an endomorphism of this line bundle over
the map $(u,[\boldsymbol z])\longmapsto
\bigl(T(u),[M(u)\boldsymbol z]\bigr)$.
After restriction to the stabilized affine image $S$, an affine-linear
base induces an automorphism whose orbits have the polynomial
presentation-growth bound of Lemma~\ref{lem:cocycle}. In a local radial
coordinate, the lifted map is the $d$th power of that coordinate times
a regular function. We use Proposition~\ref{prop:blowup} to transfer
critical absorption to this line-bundle system.

For the arithmetic input, we use Proposition~\ref{prop:controlled-place}.
For a finitely generated characteristic-zero field, it produces a complete
rank-one valued realization with finite residue field and the estimate
\eqref{eq:controlled-valuation}, which bounds the valuation of a
nonzero algebraic expression by its presentation complexity. This
replaces the number-field product-formula argument, which does not
control the valuations of varying expressions after an arbitrary
embedding of a field with transcendental elements.

In Theorem~\ref{thm:line-bundle}, we apply this estimate along the
selected base orbit. In a contracting residue tube, the radial valuation grows
at least as a positive constant times $d^n$, whereas the valuations of
nonzero coefficients in the target equations are $o(d^n)$. Hence, for
all sufficiently large indices on each analytic orbit subsequence, the
nonzero term of least radial degree has strictly smaller valuation than
every higher-degree term and cannot be canceled. We combine this
separation with $p$-adic interpolation and Strassmann's theorem to
obtain the exact return-set conclusion, and then use equivariance of
the blow-down to transfer the conclusion to the original skew product.

In Section~\ref{sec:standard-inputs}, we record the analytic and
algebraic-geometric inputs. We construct the controlled nonarchimedean
place in Section~\ref{sec:places}, and establish the analytic coordinates,
presentation growth, and line-bundle transfer theorem in
Sections~\ref{sec:local} and~\ref{sec:linebundle}. We prove the principal
results in Section~\ref{sec:higherrank}, give the Skolem--Mahler--Lech
application in Section~\ref{sec:sml}, and collect explicit examples in
Section~\ref{sec:examples}.

Unless otherwise stated, $k$ is an algebraically closed field of
characteristic zero. We may work over $\mathbb C$: every fixed
orbit problem and the finite polynomial identities expressing its
hypotheses descend to a finitely generated characteristic-zero
subfield, which embeds in $\mathbb C$. Extension of the ground field
preserves the vanishing of the target equations. In particular,
critical absorption is represented by the radical-membership
identities constructed in Lemma~\ref{lem:good-model}, and these
identities persist under field extension.

\section{Preliminaries}
\label{sec:standard-inputs}

In this section, we fix the nonarchimedean conventions and record the
analytic and algebraic-geometric inputs used below. The valuation rings
are not assumed Noetherian.

\subsection{Analytic theorems}
\label{subsec:analytic-inputs}

\begin{definition}[Valuations]\label{def:valuations}
An additive real-valued valuation on a field $\cK$ is a map
$v:\cK\to\mathbb R\cup\{+\infty\}$ such that $v(a)=+\infty$ exactly
for $a=0$, $v(ab)=v(a)+v(b)$, and
$v(a+b)\geq\min\{v(a),v(b)\}$. Its value group is
$v(\cK^\times)$. A rank-one valuation is nontrivial and has value group
contained in $\mathbb R$; it is not assumed discrete. A valuation is
trivial if it vanishes on $\cK^\times$. When $p$ is prime and $v(p)=1$,
we use the associated absolute value $|a|=p^{-v(a)}$, with $|0|=0$.
Completeness is taken with respect to this absolute value.
\end{definition}

\begin{definition}[Valuation rings and reduction]\label{def:valuation-rings}
For a complete nonarchimedean field $\cK$, write
\[
 O=\cK^\circ=\{a:|a|\leq1\},\qquad
 \mathfrak m=\cK^{\circ\circ}=\{a:|a|<1\},\qquad
 \kappa=O/\mathfrak m.
\]
Thus $O$ is the valuation ring, $\mathfrak m$ its maximal ideal, and
$\kappa$ its residue field. An element of $\cK$ is called integral if
it belongs to $O$, and an element of $O$ is a unit precisely when its
absolute value is one. For a valued field $E$, we also write $\cO_E$
for its valuation ring and $\mathfrak m_{E^\circ}=E^{\circ\circ}$ for
its maximal ideal. The $p$-adic absolute value on $\mathbb Q_p$ is
normalized by $|p|=p^{-1}$, and $\mathbb F_p$ denotes the field with
$p$ elements.
\end{definition}

\begin{definition}[Cartesian powers and polydiscs]\label{def:polydiscs}
For a set $B$, the notation $B^s$ denotes its $s$-fold Cartesian
power, with $B^0$ the singleton consisting of the empty tuple. In
particular, $\mathfrak m^s$ is a Cartesian power, not an ideal power.
On $\cK^s$ we use the maximum norm
$\|(x_1,\ldots,x_s)\|=\max_i|x_i|$, with norm zero on the empty
tuple. Then $O^s$ and $\mathfrak m^s$ are
the closed and open unit polydiscs, respectively, and
$\lambda O^s=\{x:|x_i|\leq|\lambda|\text{ for all }i\}$ for
$\lambda\in\cK^\times$. The notation $M_s(\mathfrak m)$ denotes the
$s$-by-$s$ matrices with entries in $\mathfrak m$.
\end{definition}

\begin{definition}[Formal and restricted power series]\label{def:series-algebras}
For $\boldsymbol X=(X_1,\ldots,X_s)$ and
$I=(i_1,\ldots,i_s)\in\Nzero^s$, put
$X^I=X_1^{i_1}\cdots X_s^{i_s}$ and $|I|=i_1+\cdots+i_s$.
The Tate algebra $\cK\langle\boldsymbol X\rangle$ consists of the
series $\sum_I a_I X^I$ for which $|a_I|\to0$ as $|I|\to\infty$; its
Gauss norm is $\max_I|a_I|$. The subring
$O\langle\boldsymbol X\rangle$ consists of such series with all
coefficients in $O$, and every element of this subring converges on
$O^s$. A superscript on $O\langle\boldsymbol X\rangle^s$ denotes an
$s$-tuple of series.

The formal power-series ring $O[[\boldsymbol X]]$ imposes no decay
condition on its coefficients, and
$(\boldsymbol X)O[[\boldsymbol X]]$ is its ideal of series with zero
constant term. The same notation is used for other specified variable
tuples, including $W$ and $Z$.
\end{definition}

\begin{definition}[Reduction of series and coefficient congruences]\label{def:coefficient-congruences}
Coefficientwise reduction maps $O\langle\boldsymbol X\rangle$ to a
polynomial ring over $\kappa$; a bar denotes this reduction. An
integral series map \emph{reduces to the identity} if its reduction is
the identity polynomial map. For real $c>0$,
$p^c O\langle\boldsymbol X\rangle$ denotes the set of series whose
coefficients all have valuation at least $c$; no element $p^c\in\cK$
is required. Membership in \eqref{eq:poonen-congruence} is understood
in this coefficientwise sense.
\end{definition}

\begin{theorem}[\'{E}tale Dynamical Mordell--Lang]
\label{thm:etale-dml}
Let $X$ be a quasi-projective variety over $\mathbb C$, let $f:X\to X$ be an
\'{e}tale endomorphism, let $x\in X(\mathbb C)$, and let $V\subseteq X$ be
closed.  Then $S_f(x,V)$ is a finite union of arithmetic progressions.
\end{theorem}

This is the theorem of Bell--Ghioca--Tucker
\cite[Theorem~1.3]{Bell2010}, with the return-set notation of
Definition~\ref{def:return-set}. The ground-field reduction in the
introduction gives the same conclusion over $k$. We apply it to orbit
tails contained in the zero section.

\begin{theorem}[Poonen's interpolation theorem]
\label{thm:poonen}
Let $p$ be a prime, let $(\cK,v)$ be a complete nonarchimedean valued field
with $v(p)=1$, and write $\cK^\circ$ for its valuation ring.  Let
$G=(G_1,\ldots,G_s)\in
\cK^\circ\langle X_1,\ldots,X_s\rangle^s$.  Suppose that
\begin{equation}\label{eq:poonen-congruence}
 G_i(\boldsymbol X)-X_i\in p^c
 \cK^\circ\langle\boldsymbol X\rangle
 \quad(1\leq i\leq s)
 \qquad\text{for some }c>\frac1{p-1},
\end{equation}
where membership in \eqref{eq:poonen-congruence} means that every
coefficient has valuation at least $c$, as in
Definition~\ref{def:coefficient-congruences}.  Then there is
$\mathcal G(\boldsymbol X,n) \in\cK^\circ\langle X_1,\ldots,X_s,n\rangle^s$
such that $\mathcal G(\boldsymbol X,n)=G^n(\boldsymbol X)$ for every
$n\in\Nzero$.
\end{theorem}

We use Poonen's theorem \cite[Theorem~1]{PoonenInterpolation}
to express the iterates as the values of a restricted series in one
additional variable. In $\mathcal G(\boldsymbol X,n)$, the symbol $n$
is a series variable before specialization to a nonnegative integer.

\begin{remark}[Poonen's iteration criterion]\label{rem:poonen-iteration}
For the complete valued field and integral restricted series map of
Theorem~\ref{thm:poonen}, we also use Poonen's iteration criterion
\cite[Remark~4]{PoonenInterpolation}: if $G$ reduces to the identity
modulo the maximal ideal, then a suitable $p$-power iterate satisfies
the congruence \eqref{eq:poonen-congruence}. In
Lemma~\ref{lem:local-interpolation}, the initial reduction is an
affine-linear automorphism over the finite residue field. Its finite
order gives an iterate whose reduction is the identity, after which
Poonen's criterion applies. Neither discreteness of the valuation nor
local compactness is required.
\end{remark}

\begin{theorem}[Strassmann]
\label{thm:strassmann}
Let $\cK$ be a complete nonarchimedean field and let
$h(N)=\sum_{j\geq0} c_j N^j\in\cK\langle N\rangle$ be nonzero.  Then $h$ has
only finitely many zeros in $\cK^\circ$.
\end{theorem}

For a nontrivial valuation, one-variable Weierstrass preparation
writes a nonzero restricted series as a unit in the Tate algebra times
a polynomial \cite[Section~2.2, Corollary~9]{Bosch2014}. The unit has
no zeros, so the zero set is the finite zero set of the polynomial; no
discreteness assumption is involved. For a trivial valuation,
restricted series are themselves polynomials. Consequently, in
characteristic zero, a restricted series with infinitely many
nonnegative-integer zeros is identically zero.

\subsection{Algebraic-geometric inputs}
\label{subsec:geometric-inputs}

All rings are commutative with identity. We use
$\operatorname{Spec}B$, $\cO_U$, $\Gamma(U,\cO_U)$,
$\operatorname{Frac}(B)$, the localization $B_g$, its principal open
set $D(g)$, and the radical $\sqrt J$ in their scheme-theoretic senses.
For an affine variety $S$ over $k$, its coordinate ring is denoted by
$k[S]$. An \emph{integral} algebra means a domain, rather than an
algebra integral over the base ring, and closed subschemes retain their
defining ideals rather than only their underlying sets.

For a $B$-scheme $\mathcal X$ and a homomorphism $B\to B'$, write
$\mathcal X_{B'}=\mathcal X\times_{\operatorname{Spec}B}
\operatorname{Spec}B'$. A $B$-point is a section
$\operatorname{Spec}B\to\mathcal X$, and a geometric subscript $s$
denotes base change to the algebraically closed residue field at $s$.
The terms finite type, finite, finite projective, flat, and faithfully
flat have their usual algebraic meanings. We use \emph{\'{e}tale} to
mean smooth of relative dimension zero. Generic \'{e}taleness means
\'{e}taleness after restriction to a nonempty open subset of the
target.

\begin{proposition}[Noether normalization and generic \'{e}taleness]\label{prop:normalization-input}
Noether normalization gives a finite surjective morphism from a nonempty
integral affine variety to an affine space of the same dimension
\cite[Chapter~5, Exercise~16]{AtiyahMacdonald1969}. A finite dominant
morphism between integral varieties in characteristic zero becomes
finite \'{e}tale over a nonempty open subset of the target
\cite[17.3.7 and Theorem~17.6.1]{EGAIV4}.
\end{proposition}

We apply the two assertions of Proposition~\ref{prop:normalization-input}
successively to obtain the finite \'{e}tale presentation used in
Proposition~\ref{prop:controlled-place}.

\begin{proposition}[Local Jacobian presentations]\label{prop:jacobian-input}
Over an arbitrary base ring, a smooth morphism admits local polynomial
presentations, allowing principal localizations, in which a Jacobian
minor of size equal to the number of equations is invertible;
a minor is the determinant of a square submatrix of partial
derivatives \cite[Corollary~17.12.2]{EGAIV4}.
\end{proposition}

The formulation in Proposition~\ref{prop:jacobian-input} is valid
for the non-Noetherian valuation rings occurring in
Lemma~\ref{lem:integral-tube}.

\begin{definition}[Henselian local rings]\label{def:henselian}
A local ring $(A,\mathfrak m)$ is \emph{henselian} if, for every monic
polynomial $f\in A[T]$, each simple root of its reduction
$\bar f\in(A/\mathfrak m)[T]$ lifts to a root of $f$ in $A$.
\end{definition}

\begin{proposition}[Lifting finite \'{e}tale algebras]\label{prop:etale-lifting-input}
For a henselian local ring $(A,\mathfrak m)$
(Definition~\ref{def:henselian}) with residue field $\kappa$,
reduction induces an equivalence between finite \'{e}tale $A$-algebras
and finite \'{e}tale $\kappa$-algebras
\cite[Proposition~18.5.15]{EGAIV4}. The latter are finite products of
finite separable field extensions, where separability means that
minimal polynomials have distinct roots
\cite[Proposition~17.15.11]{EGAIV4}. In particular, each residue-field
factor---one field in that product decomposition---lifts to a finite
\'{e}tale local $A$-algebra.
\end{proposition}

We use Proposition~\ref{prop:etale-lifting-input} in
Proposition~\ref{prop:controlled-place} to select the required valued
realization.

\begin{definition}[Bundles and the tautological line bundle]\label{def:bundles}
For a vector bundle $E$ of positive rank on $S$, write
$\operatorname{Tot}(E)$ for its total space and $0_S$ for its zero
section. We use the convention that $\mathbb P_S(E)$ parametrizes
lines in the fibers of $E$; the tautological line bundle is then
$\mathcal O_{\mathbb P_S(E)}(-1)$. A trivialization identifies a vector
bundle locally with a product and determines the corresponding
transition functions. The blow-up of $Y$ along a closed center $W$ is
denoted by $\operatorname{Bl}_W(Y)$.
\end{definition}

\begin{proposition}[Blow-up along the zero section]\label{prop:bundle-blowup-input}
For a vector bundle $E$ of positive rank on $S$, as in
Definition~\ref{def:bundles}, the blow-up of its total space along
the zero section is
\begin{equation}\label{eq:blowup-bundle}
 \operatorname{Bl}_{0_S}\bigl(\operatorname{Tot}(E)\bigr)
 \simeq
 \operatorname{Tot}\bigl(\mathcal O_{\mathbb P_S(E)}(-1)\bigr).
\end{equation}
The inverse image of $0_S$ is the zero section of the tautological
line bundle
\cite[Corollary~8.7.7 and Remark~8.7.8]{EGAII}.
\end{proposition}

\section{Controlled nonarchimedean realizations}
\label{sec:places}

\subsection{Presentation complexity}

For arbitrary characteristic-zero fields, our argument requires
quantitative control of valuations beyond that available over
$\overline{\mathbb Q}$.
An abstract embedding of a finitely generated field into a $p$-adic field,
as in \cite[Proposition~4.4]{Bell2010}, preserves algebraic identities but
does not bound the valuation of a varying sequence of nonzero elements. The
contracting argument requires such a bound. In
Definition~\ref{def:presentation-complexity}, we measure the degrees and
coefficient sizes of quotient presentations; in
Proposition~\ref{prop:controlled-place}, we bound valuations by this measure
in \eqref{eq:controlled-valuation}. We apply the estimate in the contracting
case of Theorem~\ref{thm:line-bundle}.

\begin{definition}[Presentation complexity]\label{def:presentation-complexity}
Let
$R=\mathbb Z[\alpha_1,\ldots,\alpha_s]$
be a finitely generated integral domain of characteristic zero, and put
$K=\operatorname{Frac}(R)$. Write
$\boldsymbol\alpha=(\alpha_1,\ldots,\alpha_s)$ and
$\boldsymbol X=(X_1,\ldots,X_s)$, so that $P(\boldsymbol\alpha)$
means evaluation of $P$ at the chosen generators. For a nonzero
polynomial $P\in\mathbb Z[X_1,\ldots,X_s]$, let $\deg P$ be its
total degree and set
\begin{equation}\label{eq:polynomial-size}
 \sigma(P)=
 (1+\deg P)\log(2+\deg P)+\log(1+\lVert P\rVert_1),
\end{equation}
where $\lVert P\rVert_1$ is the coefficient $\ell^1$-norm formed
with the ordinary absolute value on the integer coefficients, not with a
nonarchimedean absolute value. For $\xi\in K^\times$, define
its presentation complexity by
\begin{equation}\label{eq:presentation-complexity}
 \comp_R(\xi)=
 \min_{\xi=P(\boldsymbol\alpha)/Q(\boldsymbol\alpha)}
 \max\{\sigma(P),\sigma(Q)\},
\end{equation}
where $P,Q\in\mathbb Z[\boldsymbol X]$ and
$Q(\boldsymbol\alpha)\ne0$.  Put $\comp_R(0)=0$, and define the complexity
of a finite tuple to be the maximum of the complexities of its entries.
The minimum in \eqref{eq:presentation-complexity} is attained: below any
fixed bound for $\sigma$ there are only finitely many polynomials
with integer coefficients,
since both their degrees and their coefficient norms are bounded.  Every
$\xi\in K^\times$ has at least one such quotient presentation.
All logarithms are natural logarithms. The notation $\comp_R$
always refers to a fixed finite generating tuple for $R$; that tuple
is part of the chosen presentation.
\end{definition}

For a tuple $\boldsymbol\xi=(\xi_1,\ldots,\xi_t)$ in the domain of
a rational function $H$, Lemma~\ref{lem:presentation-calculus} gives
comparison estimates showing that the subsequent growth conditions are
independent of the chosen finite presentation.

\begin{lemma}[Stability of presentation complexity]
\label{lem:presentation-calculus}
Let $R=\mathbb Z[\boldsymbol\alpha]$ be as in
Definition~\ref{def:presentation-complexity}.
\begin{enumerate}[label=\textup{(\roman*)},leftmargin=*]
\item For a fixed $H\in K(Y_1,\ldots,Y_t)$ there are constants
$A_H,B_H>0$ such that, whenever $H(\boldsymbol\xi)$ is defined,
$\comp_R\bigl(H(\boldsymbol\xi)\bigr) \leq A_H\bigl(1+\comp_R(\boldsymbol\xi)\bigr)+B_H$.
\item If $R_1,R_2$ are finitely generated $\mathbb Z$-domains with
fraction field $K$, then there are constants $A,B>0$ such that
$\comp_{R_2}(\xi)\leq A\comp_{R_1}(\xi)+B\quad(\xi\in K)$.
The reverse inequality holds with possibly different constants.
\end{enumerate}
\end{lemma}

\begin{proof}
For \textup{(i)}, we first substitute any zero coordinates into $H$;
there are only finitely many possible patterns. If the resulting
function is constant, it contributes only a fixed bound; a zero value
has complexity zero. In the remaining cases, we choose presentations
$\xi_i=P_i(\boldsymbol\alpha)/Q_i(\boldsymbol\alpha)$ attaining their
complexities, and set
\[
 N=\max_i\{\deg P_i,\deg Q_i\},\qquad
 L=\max_i\{\log(1+\lVert P_i\rVert_1),
                 \log(1+\lVert Q_i\rVert_1)\}.
\]
We write $H=U/V$, clear its fixed coefficient denominators, and then
multiply the numerator and denominator by $\prod_i Q_i^D$, where
$D=\max(\deg U,\deg V)$. Submultiplicativity of the coefficient
$\ell^1$-norm gives a resulting numerator and denominator of degree
$O(N+1)$ and logarithmic coefficient norm $O(L+1)$, with constants
depending only on $H$ and the presentation of $K$. Consequently, the
resulting presentation sizes are $O((N+1)\log(N+2)+L)$, which
establishes \textup{(i)}.

For \textup{(ii)}, we write each generator of $R_1$ as a fixed quotient
in the generators of $R_2$, with a common denominator. Substitution
in a quotient presentation of degree at most $N$ and logarithmic
coefficient norm at most $L$, followed by clearing denominators,
gives polynomials of degree $O(N+1)$ and logarithmic coefficient norm
$O(L+N+1)$. Substitution in \eqref{eq:polynomial-size} gives the asserted
comparison. We obtain the reverse inequality by interchanging the two
generating tuples.
\end{proof}

In particular, for $b_n\to\infty$, the growth conditions $O(b_n)$
and $o(b_n)$ are unchanged by replacing the chosen finite presentation
of $K$.

\subsection{A finite-residue place with a quantitative bound}

\begin{definition}[Controlled places]\label{def:controlled-place}
For $R$ and $K$ as in Definition~\ref{def:presentation-complexity}, a
\emph{controlled place} is an embedding $\iota:K\hookrightarrow\cK$
into a complete rank-one valued field with finite residue field,
together with a constant $A>0$ such that
$v(\iota(\xi))\leq A\comp_R(\xi)$ for every $\xi\in K^\times$.
The valuation and residue-field conventions are those of
Definitions~\ref{def:valuations} and~\ref{def:valuation-rings}.
\end{definition}

In Proposition~\ref{prop:controlled-place}, we construct a controlled place
while preserving any denominators specified in advance. In the proof, the
\emph{content} of an integer polynomial is the positive greatest common
divisor of its coefficients.

\begin{definition}[Field norms]\label{def:field-norm}
For a finite field extension $E/D$, let $[E:D]=\dim_D E$. The field norm
$N_{E/D}(a)$ is the determinant of multiplication by $a$ on the
$D$-vector space $E$. The same determinant convention is used for a finite
free algebra over a ring.
\end{definition}

\begin{proposition}[Controlled place]\label{prop:controlled-place}
Let $R$ and $K=\operatorname{Frac}(R)$ be as in
Definition~\ref{def:presentation-complexity}. Fix
$g\in R\setminus\{0\}$ and a finite set
of rational primes.  There are an element $h\in R\setminus\{0\}$, the
localization $R'=R[1/(gh)]$, and

\begin{enumerate}[label=\textup{(\roman*)},leftmargin=*]
\item a prime $p\ge3$;
\item a complete rank-one nonarchimedean field $\mathcal K$ with valuation
$v$, normalized by $v(p)=1$, and with finite residue field;
\item an injective homomorphism $\iota:K\hookrightarrow\mathcal K$;
\end{enumerate}

such that $p$ is not excluded, $\iota(R')$ lies in the valuation ring
$\mathcal K^\circ$, and, for some constant $A>0$ depending only on the
chosen presentation and localization,
\begin{equation}\label{eq:controlled-valuation}
 v\bigl(\iota(\xi)\bigr)
 \le A\comp_R(\xi)
 \qquad(\xi\in K^\times).
\end{equation}
The elements $g$ and $h$, and hence every element of the multiplicative set
used to define $R'$, map to units of $\mathcal K^\circ$.
\end{proposition}

\begin{proof}
We divide the proof into three parts: construction of a finite \'{e}tale
model, realization of that model over a valued field with finite residue
field, and the field-norm estimate of Definition~\ref{def:field-norm}.

\smallskip
\noindent\emph{A finite \'{e}tale model.}
We apply Noether normalization to $R\otimes_{\mathbb Z}\mathbb Q$.
The resulting finite extension of rational function fields is separable
because $K$ has characteristic zero. After clearing denominators and
inverting a discriminant, we obtain from
Proposition~\ref{prop:normalization-input} algebraically independent
elements $t_1,\ldots,t_e\in R$ and nonzero elements
$c\in\mathbb Z$ and $s\in\mathbb Z[t_1,\ldots,t_e]$ such that, after
enlarging $h$,
\begin{equation}\label{eq:finite-etale-presentation}
 A=\mathbb Z[1/c,t_1,\ldots,t_e,1/s]
 \longrightarrow R'
\end{equation}
is finite \'{e}tale.
The underlying $A$-module of this finite \'{e}tale algebra is finite
projective \cite[Proposition~18.3.1(ii)]{EGAIV4}; a further principal
localization therefore makes this module free. We absorb this localization into $c,s,h$,
retaining the notation $A,R'$. The free module has positive rank;
hence \eqref{eq:finite-etale-presentation} is faithfully flat.
When an element such as $g$ is inverted on the finite algebra, we also invert
its nonzero norm on the base; Cayley--Hamilton then expresses $g^{-1}$ inside
that finite base localization.  Thus the localizations used here retain
finiteness and may be collected into the single element $h$.

We choose a rational prime $p\geq3$ outside the excluded set and not
dividing $c$ or the content of $s$.  Then $s$ remains a nonzero polynomial over
$\overline{\mathbb F}_p$, so it is nonzero at some tuple in
$\overline{\mathbb F}_p^e$.  The coordinates of that tuple belong to one
finite extension of $\mathbb F_p$.  We use these coordinates as the residues of
the $t_i$, and denote the field that they generate by $\kappa_0$.
Faithful flatness of \eqref{eq:finite-etale-presentation} then
gives a closed point
$\mathfrak m\in\operatorname{Spec}R'$
above that base point.  Its residue field
$\kappa=R'/\mathfrak m$ is a finite separable extension of $\kappa_0$ because
$R'/A$ is finite \'{e}tale; in particular, both fields are finite.
Because $\mathfrak m$ is a point of the localization $R'$, the images of
$g$ and $h$ in $\kappa$ are nonzero.

\smallskip
\noindent\emph{A valued realization.}
We let $E_0/\mathbb Q_p$ be the unramified extension with residue field
$\kappa_0$, and normalize its valuation by $v_{E_0}(p)=1$. We choose lifts
$a_i\in\cO_{E_0}$ of the images of the $t_i$
at $\mathfrak m$, and positive real numbers
$\omega_1,\ldots,\omega_e$ for which
$1,\omega_1,\ldots,\omega_e$ are linearly independent over $\mathbb Q$.
We write $\boldsymbol t=(t_1,\ldots,t_e)$,
$\boldsymbol a=(a_1,\ldots,a_e)$,
$\boldsymbol X=(X_1,\ldots,X_e)$ and
$\boldsymbol\omega=(\omega_1,\ldots,\omega_e)$; in this part of
the proof the variable tuple has length $e$. For a multi-index $I$,
$I\cdot\boldsymbol\omega=\sum_i i_i\omega_i$.
On $E_0(X_1,\ldots,X_e)$, we take the weighted Gauss valuation
\begin{equation}\label{eq:weighted-gauss}
 v_0\!\left(\sum_I b_I\boldsymbol X^I\right)
   =\min_{b_I\ne0}\bigl(v_{E_0}(b_I)+I\mathbin{\cdot}\boldsymbol\omega\bigr),
 \qquad v_0(p)=1,
\end{equation}
and embed $\mathbb Q(\boldsymbol t)$ by $t_i\mapsto a_i+X_i$.  Rational
independence makes the minimizing monomial unique. It also shows
that every element of value zero has residue in $\kappa_0$: if the leading
monomials of a numerator and denominator have equal value, they have the same
exponent and their leading coefficients have equal integral valuation.
The product of the two unique least-value monomials is the unique
least-value contribution to a product and therefore cannot cancel; hence
$v_0(fg)=v_0(f)+v_0(g)$. Comparing the minimum weights coefficientwise gives
$v_0(f+g)\geq\min\{v_0(f),v_0(g)\}$, and the valuation extends to quotients by
$v_0(f/g)=v_0(f)-v_0(g)$.
Thus the residue field of the restriction $v_0$ to
$\mathbb Q(\boldsymbol t)$ is exactly $\kappa_0$: the reverse inclusion
holds because the residues of the $t_i$ generate $\kappa_0$ over
$\mathbb F_p$.

We let $\cF$ be the completion of $\mathbb Q(\boldsymbol t)$ for $v_0$ and let
$\cF^\circ$ be its valuation ring.  Completion leaves both the value group
and the residue field unchanged: a nonzero element in the completion has
an approximation whose difference has strictly larger valuation, and hence
the same value and leading residue.  When $e=0$, the construction simply
means $\cF=\mathbb Q_p$ and $\kappa_0=\mathbb F_p$.  The elements $c$ and $s(\boldsymbol t)$
are units: $p\nmid c$, and $s(\boldsymbol a)$ has nonzero residue.  Hence
$A\to\cF^\circ$ is defined.  Base change gives a finite \'{e}tale algebra
$B=R'\otimes_A\cF^\circ$.
The chosen point $\mathfrak m$ gives $\kappa$ as one field factor
of the finite \'{e}tale $\kappa_0$-algebra
$B\otimes_{\cF^\circ}\kappa_0$.
Since $\cF$ is complete and its valuation has rank one, its
valuation ring $\cF^\circ$ is henselian
\cite[Section~4.1, pp.~85--86]{EnglerPrestel2005}.
This result applies to nondiscrete rank-one valuations as well.

By the equivalence in Proposition~\ref{prop:etale-lifting-input}
between finite \'{e}tale algebras over a henselian local ring and over
its residue field \cite[Proposition~18.5.15]{EGAIV4}, the chosen field factor
$\kappa$ lifts to a finite \'{e}tale local
$\cF^\circ$-algebra $\cO_L$, and the complementary residue
algebra lifts to a finite \'{e}tale algebra $B^{\prime}$. The
corresponding product decomposition is
\begin{equation}\label{eq:etale-factor}
 B\simeq\cO_L\times B^{\prime},
 \qquad
 \cO_L/\mathfrak m_{\cF^\circ}\cO_L\simeq\kappa.
\end{equation}

The selected factor admits an explicit description that also
identifies its valuation and norm without assuming discreteness. We put
$f=[\kappa:\kappa_0]$ and choose a monic irreducible polynomial $\bar Q$ defining this residue
extension, and lift its coefficients to a monic polynomial
$Q\in\cF^\circ[T]$.  In $L=\cF[T]/(Q)$ let $\theta$ denote the image
of $T$.  Every element has a unique expression
$x=\sum_{i<f}a_i\theta^i$.  Define
$v(x)=\min_{i<f}v_0(a_i)$.
To check multiplicativity, we divide nonzero $x,y$ by coefficients attaining
their respective minima.  Their normalized expressions lie in
$\cF^\circ[\theta]$ and have nonzero reductions in the field
$\kappa_0[T]/(\bar Q)$. The product has nonzero reduction, so
$v(xy)=v(x)+v(y)$. Thus $L$ is a domain and, because it is
finite-dimensional over the field $\cF$, is itself a field. Its valuation
ring is $\cF^\circ[\theta]$. A Cauchy sequence in $L$ is Cauchy in each
of its finitely many coefficients, and coefficientwise completeness of
$\cF$ proves completeness of $L$. Since $\bar Q$ is separable,
$Q'(\theta)$ has value zero and is a unit.  Hence
$\cF^\circ[\theta]$ is finite \'{e}tale with residue field $\kappa$,
and the same equivalence identifies it with the selected factor $\cO_L$.

The same description gives the norm identity needed for the
quantitative bound.  For $x\ne0$, we choose a coefficient $a$ with
$v_0(a)=v(x)$ and write $x=au$.
Both $u$ and $u^{-1}$ are integral, so their multiplication matrices are
inverse matrices over $\cF^\circ$ and $N_{L/\cF}(u)$ is a unit.
Consequently
\begin{equation}\label{eq:unramified-norm}
 v_0(N_{L/\cF}(x))=f\,v(x).
\end{equation}
Composing $R'\to B$ with
projection to this factor gives
$\iota:R'\longrightarrow\cO_L$
with reduction equal to the chosen closed point.

This homomorphism is injective. Indeed, its kernel has zero contraction
to $A$, meaning that its inverse image in $A$ is the zero ideal.  If it contained $0\ne b\in R'$, multiplication by $b$ on the finite
free $A$-module $R'$ would have nonzero determinant
$N_{K/\mathbb Q(\boldsymbol t)}(b)\in A$, and Cayley--Hamilton would place a
nonzero element of $A$ in the ideal $bR'$, a contradiction.  We may therefore
take $\cK=L$ and let $v$ be its unramified extension of $v_0$, still
normalized by $v(p)=1$.  Its residue field is the finite field $\kappa$, and
$g$ and $h$ reduce to nonzero elements and are therefore units.

\smallskip
\noindent\emph{The quantitative bound.}
We fix an $A$-basis of $R'$ of rank $q$.  Multiplication by $\alpha_i$ is
represented by a fixed matrix $B_i\in M_q(A)$.  We choose one common
denominator
$\Delta=c^a s(\boldsymbol t)^b$
such that every entry of every $\Delta B_i$ lies in
$\mathbb Z[\boldsymbol t]$.  If $P\in\mathbb Z[\boldsymbol X]$ has degree
$N$ and $P(\boldsymbol\alpha)\ne0$, then
\begin{equation}\label{eq:norm-determinant}
 N_{K/\mathbb Q(\boldsymbol t)}\bigl(P(\boldsymbol\alpha)\bigr)
 =\det P(B_1,\ldots,B_s)
 =\frac{H_P(\boldsymbol t)}{\Delta^{qN}}
\end{equation}
for a nonzero $H_P\in\mathbb Z[\boldsymbol t]$.
To estimate this determinant, we equip polynomial matrices with the
maximum row-sum norm formed from the coefficient $\ell^1$-norms of
their entries; this norm is submultiplicative.  We choose $C\geq1$ bounding the norms of
$\Delta I_q$ and all $\Delta B_i$, and choose $D_0$ bounding their
entry degrees.  If $P=\sum_I c_I\boldsymbol X^I$, then
\[
 S_P:=\Delta^N P(B_1,\ldots,B_s)
   =\sum_I c_I\Delta^{N-|I|}\prod_i (\Delta B_i)^{I_i}
\]
has entry degrees at most $D_0N$ and matrix norm at most
$C^N\lVert P\rVert_1$.  Since $H_P=\det S_P$, the determinant expansion
gives
$\deg H_P\leq qD_0N$ and $\lVert H_P\rVert_1\leq q!C^{qN}\lVert P\rVert_1^q$.
In particular there is a fixed constant $C_1$ such that
\begin{align}
 \deg H_P&\leq C_1(N+1),\label{eq:norm-degree}\\
 \log(1+\lVert H_P\rVert_1)
 &\leq C_1\bigl((N+1)\log(N+2)
                 +\log(1+\lVert P\rVert_1)\bigr).
 \label{eq:norm-height}
\end{align}
The slightly coarser form \eqref{eq:norm-height} matches the definition
\eqref{eq:polynomial-size} of $\sigma$; every constant is independent of $P$.

To convert these size estimates into valuation estimates, we use the
bound
\begin{equation}\label{eq:base-polynomial-bound}
 v_0\bigl(H(\boldsymbol a+\boldsymbol X)\bigr)
 \leq C_2\bigl(1+\deg H+\log(1+\lVert H\rVert_1)\bigr)
\end{equation}
for every nonzero $H\in\mathbb Z[\boldsymbol t]$, with $C_2$
independent of $H$. To prove it, we let $N=\deg H$ and select a nonzero coefficient $c_I$ with $|I|=N$.
In the translated polynomial $H(\boldsymbol a+\boldsymbol X)$, the
coefficient of $\boldsymbol X^I$ is still $c_I$: a monomial of degree at
most $N$ can contribute to that coefficient only if its exponent is $I$.
The weighted valuation \eqref{eq:weighted-gauss} therefore gives
\[
 v_0\bigl(H(\boldsymbol a+\boldsymbol X)\bigr)
 \leq v_p(c_I)+I\cdot\boldsymbol\omega
 \leq \frac{\log\lVert H\rVert_1}{\log p}
           +N\max_i\omega_i.
\]
Here $v_p(c_I)$ is the exponent of $p$ in the nonzero integer
$c_I$, and $p^{v_p(c_I)}\leq|c_I|\leq\lVert H\rVert_1$, with the
ordinary absolute value on that integer coefficient; when $e=0$, the weight
term is understood to be zero. This establishes
\eqref{eq:base-polynomial-bound} without any Diophantine approximation
condition on the chosen lifts $a_i$.

The denominator $\Delta$ in \eqref{eq:norm-determinant} is a unit at $v_0$.
Equations \eqref{eq:norm-degree}--\eqref{eq:base-polynomial-bound} therefore
give
\begin{equation}\label{eq:norm-valuation-bound}
 v_0\!\left(
 N_{K/\mathbb Q(\boldsymbol t)}(P(\boldsymbol\alpha))
 \right)\leq C_3\sigma(P).
\end{equation}
In the decomposition $B\simeq\cO_L\times B^{\prime}$ of
\eqref{eq:etale-factor},
projection to the first factor is $\iota$.
Since $P(\boldsymbol\alpha)\in R'$, multiplication by its image on
each factor preserves an integral finite free module. The complementary
summand $B^{\prime}$ is finite projective, hence free over the local ring
$\cF^\circ$.  Its determinant therefore has
nonnegative valuation, and
$v(\iota(P(\boldsymbol\alpha)))\geq0$.  If $f=[L:\cF]$, identity \eqref{eq:unramified-norm} gives
$v_0\bigl(N_{L/\cF}(\iota(P(\boldsymbol\alpha)))\bigr) =f\,v\bigl(\iota(P(\boldsymbol\alpha))\bigr)$.
We now take determinants on the product $B$ and use
\eqref{eq:norm-valuation-bound} to obtain
\[
 v\bigl(\iota(P(\boldsymbol\alpha))\bigr)
 \leq f\,v\bigl(\iota(P(\boldsymbol\alpha))\bigr)
 \leq v_0\!\left(
 N_{K/\mathbb Q(\boldsymbol t)}(P(\boldsymbol\alpha))
 \right)
 \leq C_3\sigma(P).
\]
If $\xi=P(\boldsymbol\alpha)/Q(\boldsymbol\alpha)$, then
$v(\iota(Q(\boldsymbol\alpha)))\geq0$, and hence
$v(\iota(\xi)) =v(\iota(P(\boldsymbol\alpha))) -v(\iota(Q(\boldsymbol\alpha))) \leq C_3\sigma(P)$.
We take the infimum over all presentations of $\xi$ and enlarge the constant
to absorb the positive minimum of $\sigma$ proves
\eqref{eq:controlled-valuation}.
\end{proof}

\section{\texorpdfstring{$p$}{p}-adic interpolation and cocycle complexity}
\label{sec:local}

\subsection{Residue tubes over a nondiscrete valuation ring}

\begin{definition}[Fibers, reduction, and residue tubes]\label{def:residue-tubes}
Retain $O=\mathcal K^\circ$, $\mathfrak m=\mathcal K^{\circ\circ}$,
and $\kappa=O/\mathfrak m$. For an $O$-scheme $\mathcal X$, its generic
and special fibers are $\mathcal X_{\mathcal K}$ and $\mathcal X_\kappa$.
Reduction of an $O$-point is its restriction to
$\operatorname{Spec}\kappa$; a tilde denotes reduction of points and maps.
The \emph{residue tube} of $\zeta\in\mathcal X(\kappa)$ is the set of
$O$-points reducing to $\zeta$.
\end{definition}

Lemma~\ref{lem:integral-tube} constructs coordinates on residue tubes
directly from a smooth presentation, without assuming that $O$ is Noetherian.
In its statement and proof, $W=(W_1,\ldots,W_s)$ denotes formal variables
and $w=(w_1,\ldots,w_s)$ their values. Thus $O[[W]]$ has the meaning fixed in
Definition~\ref{def:series-algebras}, and $(W)O[[W]]^r$ denotes $r$-tuples
with zero constant term. A common finite coefficient bound is uniform over
all coefficient indices.

\begin{lemma}[Integral analytic coordinates on a residue tube]
\label{lem:integral-tube}
Let $\mathcal K$ be a complete rank-one nonarchimedean field, let
$\mathcal X$ be smooth over $O$, and let $x\in\mathcal X(O)$ have reduction
$\zeta\in\mathcal X(\kappa)$.  There are an affine open neighborhood
$U$ of $\zeta$, an integer $s\geq0$, and a bijection
\begin{equation}\label{eq:residue-coordinates}
 \psi:\mathfrak m^s\longrightarrow
 \{y\in\mathcal X(O):\widetilde y=\zeta\},\qquad \psi(0)=x,
\end{equation}
with the following properties.
Every section in the target of \eqref{eq:residue-coordinates} factors through $U$.  Every regular function
$h\in\Gamma(U,\mathcal O_U)$ pulls back to a power series in $O[[W_1,
\ldots,W_s]]$, convergent on $\mathfrak m^s$.  If $U_0\subseteq U$ is
any open neighborhood of $\zeta$, the same assertion holds for functions
regular on $U_0$.  A function regular on $U_{\mathcal K}$ has a power-series
expansion whose coefficients have a common finite bound.

If $H:\mathcal X\to\mathcal X$ is an $O$-morphism fixing $\zeta$ on the
special fiber, its action on this tube has an expansion
\begin{equation}\label{eq:tube-expansion}
 \psi^{-1}H\psi(W)=e+AW+\sum_{|I|\geq2}c_IW^I,
 \qquad e\in\mathfrak m^s,\quad A\in M_s(O),\quad c_I\in O^s.
\end{equation}
If the reduction of $H$ is \'etale at $\zeta$, then $A\in\mathrm{GL}_s(O)$.
\end{lemma}

\begin{proof}
Proposition~\ref{prop:jacobian-input} supplies a local Jacobian
presentation over the arbitrary base ring $O$; see
\cite[Corollary~17.12.2]{EGAIV4}. It gives an affine neighborhood
$U=\operatorname{Spec}B$ of $\zeta$ with a presentation
\begin{equation}\label{eq:smooth-local-presentation}
 B=
 \left(
 O[T_1,\ldots,T_s,Y_1,\ldots,Y_r]/
 (F_1,\ldots,F_r)
 \right)_g
\end{equation}
for a suitable element $g$, such that
$\det\left( \frac{\partial F_i}{\partial Y_j} \right)_{1\leq i,j\leq r} \in B^\times$.
Here $s$ is the relative dimension at $\zeta$; we have shrunk
$U$ so that this dimension is constant.
Every section with reduction $\zeta$ factors through $U$:
the inverse image of $U$ is an open subset of $\operatorname{Spec}O$
containing its closed point and therefore is all of
$\operatorname{Spec}O$.  We write $x=(a,b)$ in these coordinates, with
$a=(a_1,\ldots,a_s)$ and $b=(b_1,\ldots,b_r)$, and put
$J=\left(\frac{\partial F_i}{\partial Y_j}(a,b)\right)_{i,j}$.
Then $J\in\mathrm{GL}_r(O)$.

We write $T=(T_1,\ldots,T_s)$, $Y=(Y_1,\ldots,Y_r)$, and
$F=(F_1,\ldots,F_r)$ for the coordinate and equation tuples in
\eqref{eq:smooth-local-presentation}, and consider the formal equation
\begin{equation}\label{eq:formal-tube-equations}
 F(a+W,b+\Phi(W))=0.
\end{equation}
A solution $\Phi(W)\in(W)O[[W]]^r$ is constructed recursively.
Expanding the left side gives its linear part $LW+J\Phi$ and terms of
total degree at least two in $(W,\Phi)$; here $L$ is the matrix of
partial derivatives of the equation tuple with respect to the
$T$-variables, evaluated at $(a,b)$.  Consequently, once the
homogeneous parts of $\Phi$ of degrees less than $d$ have been determined,
the part of degree $d$ is the unique solution of an equation
$J\Phi_d=-R_d$, where $R_d$ is already known.  Since $J^{-1}$ has entries
in $O$, induction constructs all coefficients in $O$ and proves formal
uniqueness.  The case $r=0$ simply omits these equations.

For $w\in\mathfrak m^s$, we set $\rho=\|w\|<1$ if $s>0$.
Each homogeneous term of degree $d$ in $\Phi(w)$ has norm at most
$\rho^d$.  Completeness therefore gives convergence and
$\Phi(w)\in\mathfrak m^r$.  Evaluation of
\eqref{eq:formal-tube-equations} shows that
$(a+w,b+\Phi(w))$ is a section with reduction $\zeta$; any localized
denominator stays a unit because its reduction is unchanged.

To prove uniqueness, suppose that $y,y'\in\mathfrak m^r$ are two solutions
of $F(a+w,b+y)=0$ with the same $w$.  Taking the difference of the two
polynomial evaluations gives
\[
 F(a+w,b+y)-F(a+w,b+y')=(J+E)(y-y'),
 \qquad E\in M_r(\mathfrak m).
\]
Indeed, every coefficient of $E$ is a sum of terms each containing at
least one coordinate of $w$, $y$, or $y'$.  The reduction of $J+E$ is the
invertible reduction of $J$, so $J+E$ is invertible over the local ring
$O$. Multiplication by $(J+E)^{-1}$ gives $y-y'=0$, hence $y=y'$.
This establishes the bijection in \eqref{eq:residue-coordinates}; its inverse
is the coordinate vector $T-a$. When $s=0$, the residue tube consists of a
single section.

Substituting $(a+W,b+\Phi(W))$ into an element of $B$ gives an integral
formal power series.  If a principal neighborhood $D(g)\subseteq U$
contains $\zeta$, then $g(x)\in O^\times$; the substituted series for
$g$ has unit constant term and hence its reciprocal also lies in
$O[[W]]$.  Evaluation of that reciprocal is valid throughout
$\mathfrak m^s$, since the nonconstant part has norm less than one
after division by its constant term.  Every open neighborhood of
$\zeta$ contains such a $D(g)$, proving the assertion for $U_0$.
Finally, $\Gamma(U_{\mathcal K},\mathcal O)=B\otimes_O\mathcal K$.
Multiplication by one nonzero element of $O$ makes any given element of
this ring integral, which proves the bounded-coefficient assertion.

We apply this construction to the components of $(T-a)\circ H$ on
$U\cap H^{-1}(U)$, an open neighborhood of $\zeta$.  Their constant
term is the coordinate vector $e$ of $H(x)$ and lies in $\mathfrak m^s$.
The resulting series is the expansion \eqref{eq:tube-expansion}. Its
linear coefficient matrix, reduced modulo $\mathfrak m$, represents the
tangent map of
$\widetilde H$ at $\zeta$ in the coordinates $T-a$ on the smooth
special fiber.  An \'etale tangent map is invertible, so the reduction of
$A$ is invertible and hence $A\in\mathrm{GL}_s(O)$.
\end{proof}

\subsection{Interpolation and exact return sets}

\begin{definition}[Analytic parametrizations]\label{def:analytic-parametrization}
A \emph{restricted analytic parametrization} of an orbit subsequence means
that, in local coordinates, its points at nonnegative integer indices
are the values of restricted power series in one variable.
\end{definition}

In Lemma~\ref{lem:local-interpolation}, we construct these parametrizations
and apply their zero sets to return times for a closed target.

\begin{lemma}[Interpolation along an \'etale reduced orbit]
\label{lem:local-interpolation}
Let $\mathcal K$ be a complete rank-one nonarchimedean field with
$v(p)=1$ and finite residue field of characteristic $p>0$.  Let
$\mathcal X$ be smooth and separated of finite type over $O$, let
$f:\mathcal X\to\mathcal X$ be an $O$-morphism, and let
$x\in\mathcal X(O)$.  Suppose that $\widetilde f$ is \'etale at every
point of the reduced forward orbit of $x$.  Then this orbit, after
discarding a finite initial segment and splitting the remaining indices
into finitely many arithmetic progressions, has one restricted analytic
parametrization (Definition~\ref{def:analytic-parametrization}) on each progression.  Consequently, for every closed
subvariety $V\subseteq\mathcal X_{\mathcal K}$, the return set
$\{n\geq0:f^n(x)\in V\}$
is a finite union of infinite arithmetic progressions and a finite set.
\end{lemma}

\begin{proof}
The set $\mathcal X(\kappa)$ is finite: a finite affine cover reduces
this assertion to finitely many choices for a finite list of coordinates
over the finite field $\kappa$.  The reduced orbit is therefore
eventually periodic.  We choose $a\geq0$ and $m\geq1$ so that
$\widetilde f^{a+b+jm}(\widetilde x)$ is independent of $j\geq0$ for
each $0\leq b<m$.  We fix $b$, put $x'=f^{a+b}(x)$ and $H=f^m$, and let
$\zeta=\widetilde{x'}$.  The map $\widetilde H$ fixes $\zeta$ and is
\'etale there, because it is a composite of maps \'etale at the
successive points of that reduced cycle.

By Lemma~\ref{lem:integral-tube}, the action of $H$ on the residue tube
of Definition~\ref{def:residue-tubes} has coordinates
\begin{equation}\label{eq:local-orbit-series}
 G(W)=e+AW+\sum_{|I|\geq2}c_IW^I,
 \qquad e\in\mathfrak m^s,\quad A\in\mathrm{GL}_s(O),\quad c_I\in O^s,
\end{equation}
and $x'$ has coordinate $0$.  If $e=0$, then $H(x')=x'$, so the corresponding subsequence is
constant and its return set to $V$ is either empty or the entire subsequence.
This case includes $s=0$.

Otherwise, we choose a nonzero component $\lambda$ of $e$ having largest
absolute value.  Then $0<|\lambda|<1$ and $e\in\lambda O^s$.  For
$W\in\lambda O^s$ the integral coefficients give
$G(W)\in\lambda O^s$.  In coordinates $W=\lambda Z$, with $Z=(Z_1,\ldots,Z_s)$, this
invariant closed polydisc has the rescaled map from
\eqref{eq:local-orbit-series}
\begin{equation}\label{eq:rescaled-orbit}
 g(Z)=\lambda^{-1}e+AZ+
      \sum_{|I|\geq2}c_I\lambda^{|I|-1}Z^I\in O\langle Z\rangle^s.
\end{equation}
Membership in \eqref{eq:rescaled-orbit} follows because the coefficients of degree $d$
have norm at most $|\lambda|^{d-1}$ for $d\geq2$. Its reduction is the
affine-linear automorphism
$\bar g(Z)=\overline{\lambda^{-1}e}+\bar A Z$ over $\kappa$.
The group of affine-linear automorphisms over $\kappa$ is finite, so
$\bar g$ has finite order $r$. Thus $g^r$ reduces to the identity as
a polynomial map.

Every coefficient of $g^r(Z)-Z$ has positive valuation. Because this
difference is a restricted series, its nonzero coefficient norms tend to
zero and therefore attain a maximum strictly smaller than one. Hence there
is $c>0$ that bounds all coefficient valuations from below, even when the
value group is dense. If $g^r(Z)=Z$, set $e_0=0$; then the zero series belongs
to $p^{c'}O\langle Z\rangle^s$ for every $c'>1/(p-1)$. If $g^r(Z)\ne Z$,
Poonen's iteration criterion (Remark~\ref{rem:poonen-iteration};
\cite[Remark~4]{PoonenInterpolation}) gives $e_0\geq0$ such that
\begin{equation}\label{eq:poonen-iterate}
 Q=g^{rp^{e_0}},\qquad
 Q(Z)-Z\in p^{c'}O\langle Z\rangle^s,\qquad c'>\frac1{p-1}.
\end{equation}
As in Definition~\ref{def:coefficient-congruences}, the notation in
\eqref{eq:poonen-iterate} imposes a bound on
coefficient valuations and does not require an element $p^{c'}$ to
belong to $\mathcal K$. Poonen's theorem
(Theorem~\ref{thm:poonen}; \cite[Theorem~1]{PoonenInterpolation}), valid over any complete valued
field with $v(p)=1$, supplies
\begin{equation}\label{eq:local-interpolant}
 \Theta(Z,N)\in O\langle Z,N\rangle^s,\qquad
 \Theta(Z,n)=Q^n(Z)\quad(n\geq0).
\end{equation}
We write $q=rp^{e_0}$.  For $0\leq\ell<q$, the point
$z_\ell=g^\ell(0)$ lies in $O^s$, and the series
$\lambda\Theta(z_\ell,N)$ from \eqref{eq:local-interpolant} parametrizes, in tube coordinates, the
subsequence $H^{\ell+qn}(x')$.  Its values lie in $\lambda O^s$ for
all $N\in O$.

We choose generators $h_1,\ldots,h_t$ of the ideal of $V\cap U_{\mathcal K}$
in the finite-type $\mathcal K$-algebra of the affine chart provided by
Lemma~\ref{lem:integral-tube}.  Their bounded expansions become
restricted series after the substitution $W=\lambda Z$, namely
$h_i(\psi(\lambda Z))\in\mathcal K\langle Z\rangle$.
Substitution $Z=\Theta(z_\ell,N)$ gives a series in
$\mathcal K\langle N\rangle$.  To see that the substitution converges,
the coefficients of the outer series tend to zero while each monomial
in the integral inner series has Gauss norm at most one; hence the
substituted tails tend to zero in the Gauss norm.

By Theorem~\ref{thm:strassmann}, each resulting one-variable series
either vanishes identically or has finitely many zeros in $O$.  If all
the compositions for a particular $\ell$ vanish identically, every
nonnegative integer on that subsequence gives a return to $V$.
Otherwise the zeros of one nonzero composition contain every return
on the subsequence and are finite.  We take the finite union over
$b$ and $\ell$ and then restore the first $a$ indices, which proves
the claim.
\end{proof}

\subsection{Projective cocycle growth}

The following estimate controls the presentation complexity of the
projective direction in the higher-rank fiber.

\begin{definition}[Projective presentation growth]\label{def:presentation-growth}
Use the complexity of Definition~\ref{def:presentation-complexity}
and the asymptotic notation of Definition~\ref{def:asymptotic-growth}.
If $X\hookrightarrow\Pj^N_K$ is a locally closed embedding, we say that a
sequence $x_n\in X(K)$ has \emph{$R$-presentation growth $O(b_n)$} if it has
homogeneous representatives $\boldsymbol x_n$, meaning nonzero
coordinate vectors with projective class $x_n$, satisfying
$\comp_R(\boldsymbol x_n)=O(b_n)$.
The embedding is always specified. Define \emph{$R$-presentation
growth $o(b_n)$} analogously, replacing $O(b_n)$ by $o(b_n)$.
\end{definition}
Lemma~\ref{lem:presentation-calculus}
shows that changing the finite presentation of $K$, or applying a fixed
rational expression defined along the sequence, changes the complexity
bound only by fixed multiplicative and additive constants.

For an affine-linear map $T$, the descending sequence of affine
subspaces $T^n(\A^m)$ stabilizes: every strict inclusion lowers the
dimension. On its stabilized image $S$, the restriction $T|_S$ is a
surjective affine-linear self-map and hence an affine automorphism.
After discarding finitely many iterates, every base orbit lies in $S$.

The \emph{projective-linear cocycle} in \eqref{eq:projective-cocycle}
records the base point
and the fiber direction updated by $M(u)$. The \emph{standard product
embedding} is $u\mapsto[1:u]$ followed by the Segre embedding, whose
coordinates are pairwise products of the homogeneous coordinates
of the two factors. The notation $\deg(T^q)$ always refers to the
degree of the composed map, not to $(\deg T)^q$.

\begin{lemma}[The projective-linear cocycle]\label{lem:cocycle}
Let $T:\A^m\to\A^m$ be an \'{e}tale polynomial endomorphism, let
$M(u)\in\operatorname{GL}_r
(K[u_1,\ldots,u_m])$, where $K$ is a finitely generated field of
characteristic zero, and put $D=\max\{1,\deg T\}$.  When $m=0$, set
$D=1$ and interpret the base as a point.  Then
\begin{equation}\label{eq:projective-cocycle}
 \widehat T:\A^m\times\Pj^{r-1}\longrightarrow
 \A^m\times\Pj^{r-1},
 \qquad
 \widehat T(u,[\boldsymbol z])=
 \bigl(T(u),[M(u)\boldsymbol z]\bigr),
\end{equation}
is \'{e}tale.  For every $K$-rational initial point, there is a finitely
generated $\mathbb Z$-domain $R\subseteq K$ with fraction field $K$ such that,
for the standard product embedding, the corresponding orbit has
$R$-presentation growth
\begin{equation}\label{eq:cocycle-growth}
 \begin{cases}
 O\bigl(n^2\log(n+2)\bigr),&D=1,\\
 O(nD^n),&D\geq2.
 \end{cases}
\end{equation}
If $T$ is an automorphism, then $\widehat T$ is an automorphism.
\end{lemma}

\begin{proof}
We factor $\widehat T$ as
\begin{equation}\label{eq:cocycle-factorization}
 (u,[\boldsymbol z])\longmapsto
 (u,[M(u)\boldsymbol z])\longmapsto
 (T(u),[M(u)\boldsymbol z]).
\end{equation}
The first arrow in \eqref{eq:cocycle-factorization} is an automorphism over $\A^m$, with inverse induced by
$M(u)^{-1}$, and the second is the base change of $T$.  Hence $\widehat T$ is
\'{e}tale.  If $T$ is an automorphism, the inverse is explicitly
$(u,[\boldsymbol z])\longmapsto \bigl(T^{-1}(u),[M(T^{-1}(u))^{-1}\boldsymbol z]\bigr)$.
It is regular because both $T^{-1}$ and $M^{-1}$ are polynomial.

For the presentation-growth estimate, we enlarge and localize a finitely
generated $\mathbb Z$-domain $R=\mathbb Z[\boldsymbol\alpha]$ so that the
coefficients
of $T$ and $M$, the affine coordinates of $u_0$, and a nonzero vector
representative $\boldsymbol z_0$ all lie in $R$.  We fix polynomial representatives in $\mathbb Z[\boldsymbol X]$ for
these finitely many elements and set
$u_n=T^n(u_0)$ and $\boldsymbol z_{n+1}=M(u_n)\boldsymbol z_n$.

For a tuple of polynomial representatives, we let $\delta$ be the largest
degree in $\boldsymbol X$ and let $\eta$ be the largest value of
$\log(1+\lVert\cdot\rVert_1)$.  We denote the corresponding quantities for
$u_n$ by $\delta_n,\eta_n$; when $m=0$, set both quantities equal to
zero and omit the empty base recurrence.  Estimating degrees and coefficient norms under substitution gives a
constant $C$ independent of $n$ such that
\begin{equation}\label{eq:base-complexity-recurrence}
 \delta_{n+1}\leq D\delta_n+C,
 \qquad
 \eta_{n+1}\leq D\eta_n+C.
\end{equation}
Indeed, each coordinate of $T$ is a fixed sum of monomials of degree at most
$D$; degrees multiply by at most $D$, while $\ell^1$-norms multiply under
products and add over this fixed finite sum.  Solving
\eqref{eq:base-complexity-recurrence} yields
\begin{equation}\label{eq:base-complexity-solution}
 (\delta_n,\eta_n)=
 \begin{cases}
 O(n),&D=1,\\
 O(D^n),&D\geq2.
 \end{cases}
\end{equation}

We let $\delta'_n,\eta'_n$ be the analogous quantities for
$\boldsymbol z_n$.  Evaluation of the fixed polynomial matrix $M$ at $u_n$
and multiplication of the resulting matrix by $\boldsymbol z_n$ give
\begin{equation}\label{eq:fiber-complexity-recurrence}
 \delta'_{n+1}\leq\delta'_n+C(\delta_n+1),
 \qquad
 \eta'_{n+1}\leq\eta'_n+C(\eta_n+1).
\end{equation}
Summing the bounds in \eqref{eq:fiber-complexity-recurrence} gives
\begin{equation}\label{eq:fiber-complexity-solution}
 (\delta'_n,\eta'_n)=
 \begin{cases}
 O(n^2),&D=1,\\
 O(D^n),&D\geq2.
 \end{cases}
\end{equation}
Finally, for a representative of degree $\delta$ and logarithmic
$\ell^1$-norm $\eta$, the definition \eqref{eq:polynomial-size} of $\sigma$ gives
$\sigma=O\bigl((1+\delta)\log(2+\delta)+\eta\bigr)$.
We use the affine representatives $[1:u_n]$ and the vector representatives
$\boldsymbol z_n$, followed by the Segre embedding of the product.  Equations
\eqref{eq:base-complexity-solution} and
\eqref{eq:fiber-complexity-solution} give
$O(n^2\log(n+2))$ when $D=1$ and $O(nD^n)$ when $D\geq2$, which is
\eqref{eq:cocycle-growth}.
\end{proof}

\section{A transfer theorem for superattracting line-bundle dynamics}
\label{sec:linebundle}

\subsection{Line-bundle maps and integral models}

We blow up the invariant zero section to replace the higher-rank fiber by
a line bundle over the projective directions. In Theorem~\ref{thm:line-bundle},
we prove a transfer result for this line-bundle dynamics under contraction,
critical absorption, and an orbitwise presentation-growth hypothesis; its
statement does not depend on the coordinate form \eqref{eq:intro-map}.

\begin{definition}[Line-bundle dynamics and the vertical differential]\label{def:line-bundle-dynamics}
Let $X$ be a smooth quasi-projective variety, let $\varphi:X\to X$ be an
\'{e}tale endomorphism, and let $\pi:L\to X$ be a line bundle with zero
section $Z_L$. An endomorphism $\Phi:L\to L$ lies over $\varphi$ if
$\pi\circ\Phi=\varphi\circ\pi$. Let $T_{L/X}$ denote the relative
tangent bundle. The vertical differential is the line-bundle morphism
$d_{\mathrm{vert}}\Phi:T_{L/X}\longrightarrow\Phi^*T_{L/X}$,
and $C_\Phi$ denotes its zero scheme. Locally in source and target
trivializations, this zero scheme is defined by the coefficient of the
derivative in the fiber coordinate. For a closed subscheme $W$, write
$\cI_W$ for its ideal sheaf; $\Phi^*\cI_W$ is the ideal generated by
pullbacks of its local sections, and $\cI_W^d$ is its $d$th ideal power.
The superattracting condition used below is the ideal inclusion
\eqref{eq:ideal-contraction}.
\end{definition}

For a homogeneous polynomial $h$, let $D_+(h)$ denote its nonvanishing
locus in projective space. In Lemma~\ref{lem:adapted-projective-charts}, a
cover is \emph{subordinate} to a prescribed cover if each of its members is
contained in a member of the prescribed cover. The notation $L|_U$ denotes
the restriction of $L$ to $U$.

\begin{lemma}[Projective charts adapted to a line bundle and a target]
\label{lem:adapted-projective-charts}
Let $K$ be a field, let $X\hookrightarrow\Pj^N_K$ be a locally closed
embedding of a quasi-projective variety, and let $\pi:L\to X$ be a line
bundle.  Let $V\subseteq L$ be closed.  Write $\overline X$ for the
projective closure of $X$ in this embedding.  There is a finite affine
open cover $X=\bigcup_{\alpha}U_\alpha$ with the following properties.
For each $\alpha$, there is a homogeneous polynomial $h_\alpha$ of
positive degree such that
$U_\alpha=\overline X\cap D_+(h_\alpha)\subseteq X$,
the line bundle is trivial on $U_\alpha$, and, in a chosen fiber
coordinate $t_\alpha$, the intersection $V\cap L|_{U_\alpha}$ is cut out
by finitely many equations
\begin{equation}\label{eq:target-equations}
 P_{\alpha,i}(x,t_\alpha)
   =\sum_{j=0}^{e_{\alpha,i}}a_{\alpha,i,j}(x)t_\alpha^j.
\end{equation}
Each nonzero coefficient has a fixed presentation
\begin{equation}\label{eq:target-coefficients}
 a_{\alpha,i,j}
   =\frac{A_{\alpha,i,j}}{h_\alpha^{b_{\alpha,i,j}}},
 \qquad
 \deg A_{\alpha,i,j}
   =b_{\alpha,i,j}\deg h_\alpha,
\end{equation}
where $A_{\alpha,i,j}$ is homogeneous and $b_{\alpha,i,j}\geq0$.
The cover may also be chosen subordinate to any prescribed finite open
cover, with the same conclusion for any prescribed finite collection of
regular functions on its members after restriction.
\end{lemma}

\begin{proof}
The sets $D_+(h)\cap\overline X$, with $h$ homogeneous of
positive degree, form an affine basis for the topology of
$\overline X$. More precisely, if $B$ is its homogeneous
coordinate ring, then
\begin{equation}\label{eq:projective-chart-ring}
 D_+(h)\cap\overline X
 \simeq \operatorname{Spec}B_{(h)},
 \qquad
 B_{(h)}=\bigl(B[h^{-1}]\bigr)_0,
\end{equation}
where the subscript $0$ denotes the degree-zero part, consisting
of homogeneous fractions whose numerator and denominator have
the same degree.
These are the affine charts in the construction of
$\operatorname{Proj}$
\cite[Propositions~2.3.4 and~2.3.6, and~2.4.1]{EGAII}.
Because $X$ is open in $\overline X$, a bundle-trivializing cover and
any prescribed cover may be refined by affine opens of the form
$D_+(h)\cap\overline X$ contained in $X$. Quasi-compactness gives a finite
subcover.

For $B=K[X_0,\ldots,X_N]/I(\overline X)$, the chart description \eqref{eq:projective-chart-ring} gives
$\Gamma(U_\alpha,\cO_X)=B_{(h_\alpha)}$.
Every element is represented by $A/h_\alpha^b$, where $A$ is homogeneous
of degree $b\deg h_\alpha$: we combine its finitely many homogeneous
fractions over a common denominator and lift the numerator from $B$
to the polynomial ring.  Finally, we use
$\Gamma(L|_{U_\alpha},\cO_L)=B_{(h_\alpha)}[t_\alpha]$
and the fact that this ring is Noetherian.  Thus the ideal of the closed target on this affine chart
has finitely many polynomial generators.  We apply the preceding
description to their coefficients to prove the assertions.
\end{proof}

The transfer argument requires an integral model in which contraction
and critical absorption persist in every special fiber.

\begin{definition}[Models and spreading out]\label{def:models}
A \emph{model over $R$} is a collection of schemes and morphisms over $R$
whose base change to $\operatorname{Frac}(R)$ recovers the given data.
Choosing such a model over a finitely generated ring is called
\emph{spreading out}.
\end{definition}

Locally, the inclusion
$\Phi^*\cI_{Z_L}\subseteq\cI_{Z_L}^d$ states that a target fiber
coordinate pulls back to a multiple of the $d$th power of a source fiber
coordinate. In Lemma~\ref{lem:good-model}, a subscript $R$ denotes the
chosen model and a subscript $s$ its base change to a geometric point.

\begin{lemma}[A model preserving contraction and critical absorption]
\label{lem:good-model}
Let $K$ be a finitely generated field of characteristic zero.  Suppose that
$X$, $L$, $\varphi$, $\Phi$, $Z_L$, $C_\Phi$, a point $y_0\in L(K)$, and a
closed subvariety $V\subseteq L$ are in the setting of
Definition~\ref{def:line-bundle-dynamics}. Assume
$\Phi^*\cI_{Z_L}\subseteq\cI_{Z_L}^d$ and $\Phi^\ell(C_\Phi)\subseteq Z_L$ set-theoretically.
Then there are a finitely generated integral $\mathbb Z$-algebra $R$ with
fraction field $K$ and models of all these data over $R$ such that:
\begin{enumerate}[label=\textup{(\roman*)},leftmargin=*]
\item $X_R$ is smooth over $R$, $L_R\to X_R$ is a line bundle, and
$\varphi_R$ is \'{e}tale;
\item $y_0$ extends to an $R$-point, $V$ extends to a closed subscheme, and
$\Phi_R^*\cI_{Z_{L,R}}\subseteq\cI_{Z_{L,R}}^d$;
\item if $C_{\Phi,R}$ is the zero scheme of the relative vertical
differential, then for every geometric point $s\to\operatorname{Spec}R$,
$\Phi_s^\ell(C_{\Phi,s})\subseteq Z_{L,s}$ set-theoretically.
\end{enumerate}
Any prescribed finite collection of nonzero elements of $K$ may, after
enlarging $R$, be required to be units.
\end{lemma}

\begin{proof}
All schemes, morphisms, the line bundle, its zero section, the point, and the
target involve finitely many equations and transition functions.  We
therefore descend them to a finitely generated integral $\mathbb Z$-algebra
with fraction field $K$.  Smoothness of $X$ and \'{e}taleness of $\varphi$
hold on open neighborhoods of the generic fiber, so we localize the base to
make them hold everywhere. After one further localization, the point extends to
an $R$-point. This establishes \textup{(i)} and the first assertion of
\textup{(ii)}.

The ideal inclusion for radial contraction is a finite collection of
identities on a finite affine trivializing cover of $L_R$.  We clear their
finitely many denominators and localize once more; the same identities then
prove
$\Phi_R^*\cI_{Z_{L,R}}\subseteq\cI_{Z_{L,R}}^d$ over $R$.

It remains to preserve the set-theoretic inclusion in every fiber.  On a
finite affine cover, we let $a_1,\ldots,a_t$ be generators of
$(\Phi^\ell)^*\cI_{Z_L}$ and let $J$ define $C_\Phi$.  The generic-fiber
inclusion says
$a_i\in\sqrt J\quad(1\leq i\leq t)$.
Thus, after choosing one common exponent $N\geq1$, we obtain finite identities
$a_i^N\in J$.  We spread the vertical differential and these identities to $R$,
again clearing finitely many denominators.  If a geometric point of a special
fiber lies in $C_{\Phi,s}$, every element of $J_s$ vanishes there; hence every
$a_i^N$ and therefore every $a_i$ vanishes there.  Its image under
$\Phi_s^\ell$ lies in $Z_{L,s}$, which proves \textup{(iii)}. We obtain
the final assertion by adjoining the inverses of the prescribed finite set.
\end{proof}

\subsection{The transfer theorem}

We impose the presentation-growth hypothesis in
Theorem~\ref{thm:line-bundle} only along the selected base orbit. In its
proof, a bar denotes reduction in the chosen integral model; $\overline X$ in
Lemma~\ref{lem:adapted-projective-charts} continues to denote projective
closure. We separate the cases in which the orbit enters the zero
section, its reduction avoids the zero section, and only its reduction enters
the zero section. The final case requires comparison of the radial valuation
with the valuations of the target coefficients.

\begin{theorem}[Orbitwise line-bundle transfer]\label{thm:line-bundle}
In the preceding setting, assume that there is an integer $d\geq2$ such
that
\begin{equation}\label{eq:ideal-contraction}
 \Phi^*\cI_{Z_L}\subseteq\cI_{Z_L}^{d},
\end{equation}
and assume that $\Phi^\ell(C_\Phi)\subseteq Z_L$ set-theoretically for some
$\ell\geq1$.
Fix $y_0\in L(k)$ and a closed subvariety $V\subseteq L$, and put
$x_n=\pi(\Phi^n(y_0))$.  Descend these data to a finitely generated field
$K\subseteq k$ and choose a finitely generated $\mathbb Z$-domain
$R\subseteq K$ with fraction field $K$.  If, for one locally closed
embedding $X\hookrightarrow\Pj^N_K$, the sequence $(x_n)$ has
$R$-presentation growth $o(d^n)$, then
$\{n\in\Nzero:\Phi^n(y_0)\in V\}$
is a finite union of arithmetic progressions.
\end{theorem}

\begin{proof}
We write $y_n=\Phi^n(y_0)$ and
$\mathcal R=\{n\in\Nzero:y_n\in V\}$.

\smallskip
\noindent\emph{Choosing the charts and the nonarchimedean model.}
We apply Lemma~\ref{lem:adapted-projective-charts} to the embedding in the
presentation-growth hypothesis.  Over $K$, we fix its finite cover, bundle
trivializations, target equations, and homogeneous presentations of all
coefficients.  We spread this entire finite collection together with the
dynamical data by Lemma~\ref{lem:good-model}.  After enlarging and localizing
$R$, the chosen opens cover $X_R$, the trivializations and their inverse
transition identities hold over $R$, and the target equations on every
chart define $V_R$.  All coefficient-presentation identities also hold
over $R$.  These requirements involve only finitely many equations and
denominators.

A single localization of the base makes these generic-fiber
conditions hold over the model. Indeed, the loci where they fail
are finitely many finite-type loci with empty generic fiber. On each member
$\operatorname{Spec}B$ of a finite affine cover of such a locus, the equality $B\otimes_R K=0$
means that some nonzero element of $R$ annihilates $1$ in $B$.
We remove the locus by inverting the product of these finitely many elements.
This observation preserves the prescribed open cover, its containments,
and the generic-fiber identifications after spreading.
Lemma~\ref{lem:presentation-calculus}(ii) shows that replacing $R$ in
this way preserves the hypothesis of $o(d^n)$ presentation growth.
We choose the nonarchimedean place after these charts and identities
have been fixed. In every geometric special fiber, we have
\begin{equation}\label{eq:reduced-absorption}
 \overline y\in C_{\overline\Phi}
 \quad\Longrightarrow\quad
 \overline\Phi^{\,\ell}(\overline y)\in\overline Z_L.
\end{equation}

We apply Proposition~\ref{prop:controlled-place} to the single product of these
finitely many constant denominators, excluding $2$ and the primes dividing
$d$, and obtain an embedding $K\hookrightarrow\cK$ into a complete rank-one valued field with
finite residue field and a valuation normalized by $v(p)=1$.  The integral
model, $y_0$, and therefore every $y_n=\Phi^n(y_0)$ are
$\cK^\circ$-integral.  A transition function and its inverse are both
integral on an overlap, so their values are units.  Moreover,
\begin{equation}\label{eq:complexity-to-valuation}
 v(\xi)\leq A\comp_R(\xi)\qquad(\xi\in K^\times).
\end{equation}

\smallskip
\noindent\emph{Orbits that enter the zero section.}
If $y_a\in Z_L$ for some $a$, then every subsequent point lies in $Z_L$ by
\eqref{eq:ideal-contraction}: in a local trivialization, the pullback of a
target fiber coordinate belongs to $(t^d)$ and therefore vanishes at $t=0$.
Under the identification $Z_L\simeq X$, the
restricted map is $\varphi$.  We apply Lemma~\ref{lem:local-interpolation}
to the \'{e}tale orbit of $x_a$ and the closed subvariety
$V\cap Z_L$ (viewed as a subvariety of $X$) to obtain the required description
of the return times $n\geq a$.
The finitely many earlier indices do not affect the conclusion.  We may
therefore assume that $y_n\notin Z_L$ for every $n$.

\smallskip
\noindent\emph{Reduced orbits that avoid the zero section.}
We first consider the case in which $\overline y_n\notin\overline Z_L$ for every $n$.  If a
point of the reduced orbit belonged to $C_{\overline\Phi}$, implication
\eqref{eq:reduced-absorption} would put a later point in
$\overline Z_L$, a contradiction.  Thus the vertical differential is
invertible along the reduced orbit.  The base map is \'{e}tale, and in local
bundle coordinates $(x,t)$ the Jacobian of $\Phi$ is
$\begin{psmallmatrix}D\varphi&0\\ *&\partial t'/\partial t\end{psmallmatrix}$.
Both diagonal blocks are invertible along the reduced orbit.  Hence
$\overline\Phi$ is \'{e}tale at every point of the reduced orbit.
We now apply Lemma~\ref{lem:local-interpolation} on the total space of
the line bundle, which proves the assertion in this case.

\smallskip
\noindent\emph{Reduced orbits that enter the zero section.}
We now suppose that $\overline y_a\in\overline Z_L$ for some $a$.
The reduced zero section is invariant, so the same holds for every $n\geq a$.
We choose local trivializations of the integral line bundle around the finitely
many points of the reduced base orbit.  If $t_n$ denotes the corresponding
fiber coordinate, its valuation
$w_n=v(t_n)$
is independent of the chosen trivialization because transition functions
are units.  We have $w_n>0$ for $n\geq a$.  The ideal containment
\eqref{eq:ideal-contraction} says in these source and target trivializations
that $t_{n+1}=t_n^d h_n$ for an integral value $h_n$.  Hence
\begin{equation}\label{eq:radial-growth}
 w_{n+1}\geq d w_n,
 \qquad
 w_n\geq d^{n-a}w_a.
\end{equation}

The reduced base orbit is finite.  We partition the tail according to the
reduced base point, and choose for each such point one of the already fixed affine charts
$U_R$ containing it; write $U$ for its generic fiber.  Every integral point
in its residue tube factors through $U_R$, because the inverse image of
$U_R$ is an open subset of the spectrum of a local ring containing its
closed point.  On one fixed part of this partition, a finite collection of
equations
\begin{equation}\label{eq:transfer-target-equations}
 P_i(x,t)=\sum_{j=0}^{D_i}a_{i,j}(x)t^j,
 \qquad 1\leq i\leq s,
\end{equation}
cuts out $V$ on $L|_U$.  These are the equations \eqref{eq:target-equations} fixed before choosing
the place; their coefficient functions extend regularly over $U_R$.
Consequently
\begin{equation}\label{eq:coefficient-integrality}
 v(a_{i,j}(x_n))\geq0
\end{equation}
whenever $x_n$ lies in the chosen tube.

We refine the partition by the analytic parametrizations of the base orbit
constructed in Lemma~\ref{lem:local-interpolation} and by the reduced base
point. Along each parametrization, every $a_{i,j}$ gives a restricted
one-variable analytic function. By Strassmann's theorem
(Theorem~\ref{thm:strassmann}), this function either vanishes
identically or has finitely many zeros. Since there are only
finitely many coefficients and parametrizations, discarding a sufficiently
long initial segment leaves finitely many progressions on each of which every
$a_{i,j}(x_n)$ is either always zero or never zero.

Each $a_{i,j}$ has the fixed homogeneous quotient presentation
\eqref{eq:target-coefficients} of Lemma~\ref{lem:adapted-projective-charts}, with denominator nonzero on $U$.
We evaluate that quotient on the homogeneous representatives of $x_n$ from
the growth hypothesis.  Its numerator and denominator have the same degree,
so the result is independent of their common scaling.
Lemma~\ref{lem:presentation-calculus}(i) therefore gives
$\comp_R(a_{i,j}(x_n))=o(d^n)$ on these indices.
Hence
\eqref{eq:complexity-to-valuation} yields
\begin{equation}\label{eq:coefficient-small}
 v(a_{i,j}(x_n))=o(d^n)
\end{equation}
whenever $a_{i,j}(x_n)\ne0$.

We fix one of the resulting progressions and one equation $P_i$ from
\eqref{eq:transfer-target-equations}.  If every
coefficient sequence is identically zero, then $P_i(y_n)=0$ throughout the
progression.  Otherwise, let $j_0$ be the least index whose coefficient sequence
is nonzero on the progression. All coefficients with $j<j_0$ vanish
on that progression, whereas for $j>j_0$ with
$a_{i,j}(x_n)\ne0$ one has
\begin{align*}
 &v\bigl(a_{i,j}(x_n)t_n^j\bigr)
 -v\bigl(a_{i,j_0}(x_n)t_n^{j_0}\bigr)\\
 &\qquad\geq (j-j_0)w_n-v(a_{i,j_0}(x_n))>0
\end{align*}
for all sufficiently large $n$, by
\eqref{eq:coefficient-integrality}, \eqref{eq:radial-growth}, and
\eqref{eq:coefficient-small}.  Indeed, $w_n\geq d^{n-a}w_a$ is a positive
constant times $d^n$, whereas the subtracted term is $o(d^n)$.  Thus the
$j_0$-term is the unique summand of least valuation, and
$P_i(y_n)\ne0$.

Consequently, on the tail of each progression, every defining equation of
$V$ is either always zero or always nonzero.  Membership in $V$ is therefore
constant on that tail.  We take the finite union over all progressions and
restore the omitted indices, which proves the assertion.
\end{proof}

\begin{remark}\label{rem:complexity-gap}
The presentation-growth hypothesis is used only in the contracting
residue tube. It expresses the separation required by our method:
the algebraic complexity of the base orbit, tangent to the zero
section, must be asymptotically smaller than the transverse
contraction scale $d^n$.  Over
$\overline{\mathbb Q}$, the product formula turns a corresponding height
bound into the needed valuation estimate.  Over a general
characteristic-zero field, Proposition~\ref{prop:controlled-place} performs
that role directly.
\end{remark}

\section{Higher-rank radial skew products}
\label{sec:higherrank}

\subsection{The radial lift and its critical locus}

By Proposition~\ref{prop:bundle-blowup-input} and
\eqref{eq:blowup-bundle}, we identify the blow-up along the zero section with the total
space of the tautological line bundle
\cite[Corollary~8.7.7 and Remark~8.7.8]{EGAII}.
In Proposition~\ref{prop:blowup}, we verify the geometric hypotheses of
Theorem~\ref{thm:line-bundle} for the radial dynamics.
\begin{definition}[Lifts and equivariance]\label{def:equivariance}
A \emph{lift} to a blow-up is a map commuting with the projection to the
original space. This commutation is called \emph{equivariance}, and the
projection is called the \emph{blow-down}.
\end{definition}
\begin{definition}[Rees algebras]\label{def:rees-algebra}
The Rees algebra of an ideal $I$ is the graded algebra
$\bigoplus_{n\geq0}I^n$, with $I^n$ placed in degree $n$ and
multiplication induced by $I^a I^b\subseteq I^{a+b}$.
\end{definition}

The bundle
$\mathcal O_{S\times\Pj^{r-1}}(-1)$ in the proposition is the
tautological line bundle pulled back from $\Pj^{r-1}$.
The ring in the condition on $\rho$ is $k[S][z_1,\ldots,z_r]$.
In this section, $\widetilde Y$ and $\widetilde F$ denote a
blow-up and its lifted map, not reduction modulo a valuation.

\begin{proposition}[Radial blow-up]\label{prop:blowup}
Let $S$ be an affine space, let $T:S\to S$ be an \'{e}tale polynomial
endomorphism, and
let $M(u)\in\operatorname{GL}_r(k[S])$.  Let
$d\geq2$ and
$0\ne\rho(u,\boldsymbol z)\in(z_1,\ldots,z_r)^{d-1}$.
Define
$F(u,\boldsymbol z)= \bigl(T(u),\rho(u,\boldsymbol z)M(u)\boldsymbol z\bigr)$,
and let $\beta:\widetilde Y\to S\times\A^r$ be the blow-up of
$Z=S\times\{0\}$.  Then:
\begin{enumerate}[label=\textup{(\roman*)},leftmargin=*]
\item $\widetilde Y$ is the total space of
$\mathcal O_{S\times\Pj^{r-1}}(-1)$, and its inverse-image divisor $E=\beta^{-1}(Z)$ is
the zero section;
\item $F$ lifts to an endomorphism
$\widetilde F:\widetilde Y\to\widetilde Y$ lying over the \'{e}tale map
$\widehat T(u,[\boldsymbol z])= \bigl(T(u),[M(u)\boldsymbol z]\bigr)$;
\item $\widetilde F^*\cI_E\subseteq\cI_E^{d}$;
\item if the vertical critical locus $C$ of $F$ satisfies
$F^\ell(C)\subseteq Z$, then the vertical critical locus
$C_{\widetilde F}$ satisfies
$\widetilde F^\ell(C_{\widetilde F})\subseteq E$.
\end{enumerate}
\end{proposition}

\begin{proof}
We apply Proposition~\ref{prop:bundle-blowup-input} through the Rees
algebra of Definition~\ref{def:rees-algebra} to identify the blow-up with the
tautological line bundle and prove \textup{(i)}. Concretely, a point of
$\widetilde Y$ consists of $(u,\boldsymbol z)$ together with a projective
direction $[\boldsymbol q]$ satisfying
$\boldsymbol z\in k\boldsymbol q$; the scalar of
proportionality is the fiber coordinate of the tautological line bundle.

On the chart $q_i=1$, we write $\boldsymbol z=t\boldsymbol q$.  Since $\rho$ has fiber order
at least $d-1$, there is a polynomial $G$ in the chart coordinates such that
$\rho(u,t\boldsymbol q)=t^{d-1}G(u,\boldsymbol q,t)$.
Therefore
\begin{equation}\label{eq:lifted-map-local}
 \rho(u,t\boldsymbol q)M(u)(t\boldsymbol q)
 =t^dG(u,\boldsymbol q,t)M(u)\boldsymbol q.
\end{equation}
Because $M(u)\boldsymbol q$ is never the zero vector, its projective class
is defined on every chart.  We therefore glue formula
\eqref{eq:lifted-map-local} to a morphism $\widetilde F$.  Its projective component is $\widehat T$, and its normal component
belongs to the ideal generated by $t^d$. These two conclusions establish
\textup{(ii)} and \textup{(iii)}, respectively.

The base map $\widehat T$ is \'{e}tale by Lemma~\ref{lem:cocycle}.
We verify critical absorption by computing the radial derivative.
On a source projective chart, we write $\boldsymbol z=t\boldsymbol q$ with
$q_i=1$ and choose a target chart on which
$a(u,\boldsymbol q)=(M(u)\boldsymbol q)_j\ne0$.
Such target charts cover the source, because $M(u)$ is invertible and
$\boldsymbol q$ is nonzero.  On this overlap the lifted map is
\begin{equation}\label{eq:lift-chart}
 u'=T(u),\qquad
 \boldsymbol q'=
       \frac{M(u)\boldsymbol q}{a(u,\boldsymbol q)},\qquad
 t'=a(u,\boldsymbol q)t\rho(u,t\boldsymbol q).
\end{equation}
The factor $a(u,\boldsymbol q)$ is invertible on the overlap and is
independent of $t$. We differentiate the fiber coordinate in
\eqref{eq:lift-chart} only in the line-bundle direction and obtain
\begin{equation}\label{eq:critical-comparison}
 \frac{\partial t'}{\partial t}
   =a(u,\boldsymbol q)
       \bigl(\rho+E_{\boldsymbol z}\rho\bigr)(u,t\boldsymbol q).
\end{equation}
Here $t\,\partial_t\rho(u,t\boldsymbol q)
 =(E_{\boldsymbol z}\rho)(u,t\boldsymbol q)$.
Because $a(u,\boldsymbol q)$ is a unit,
\eqref{eq:critical-comparison} shows that
$y\in C_{\widetilde F}$ implies
$(\rho+E_{\boldsymbol z}\rho)(u,t\boldsymbol q)=0$.
The determinant identity \eqref{eq:general-critical-determinant} then gives
$\beta(y)\in C$. This calculation includes points with $t\ne0$ and
$\rho(u,t\boldsymbol q)=0$, whose images lie above the zero section even
though the target blow-down is not invertible there.

For $y\in C_{\widetilde F}$, equivariance
(Definition~\ref{def:equivariance}) and the assumed absorption
give
$\beta\bigl(\widetilde F^\ell(y)\bigr) =F^\ell\bigl(\beta(y)\bigr)\in Z$.
The equality $\beta^{-1}(Z)=E$ therefore implies
$\widetilde F^\ell(C_{\widetilde F})\subseteq E$.
\end{proof}

When $r=1$, the center is an effective Cartier divisor: locally,
its ideal is generated by one non-zero-divisor. The blow-up is
therefore the identity and the tautological line bundle is trivial.
Thus $E$ is the inverse-image divisor of $Z$; it is an exceptional
divisor only for $r\geq2$, when it records the projective directions
replacing each point of the center.

\subsection{Proof of the principal theorem}

\begin{proof}[Proof of Theorem~\ref{thm:main}]
We fix $x_0\in\A^{m+r}(k)$ and a closed subvariety
$V\subseteq\A^{m+r}$. We discard the first $N$ iterates and replace the target by its
intersection with $S\times\A^r$, where $S=T^N(\A^m)=T^{N+1}(\A^m)$.
The restriction $T|_S$ is a surjective affine-linear self-map and hence an
affine automorphism. This replacement translates the remaining return times
by $N$ and changes only finitely many initial indices. We reuse $x_0$
for the first point of the remaining orbit.

If the restriction of $\rho$ to $S\times\A^r$ is zero, then one further
iterate sends the entire space into $Z$.  The subsequent dynamics is the \'{e}tale
map $T|_S$ on the zero section.  We apply
Theorem~\ref{thm:etale-dml} to $V\cap Z$, which proves the result in this case.
We may therefore assume that this restriction of $\rho$ is nonzero.

We let $\beta:\widetilde Y\to S\times\A^r$ and $\widetilde F$ be as in
Proposition~\ref{prop:blowup}.  The base of the tautological line bundle is
$S\times\Pj^{r-1}$, and its induced endomorphism is $\widehat T$.
We choose $\widetilde x_0\in\beta^{-1}(x_0)$ and put
$\widetilde V=\beta^{-1}(V)$; if $x_0\in Z$, any projective direction
determines a lift.  After descending these data to a finitely generated
field, Lemma~\ref{lem:cocycle} shows that
$\widehat T$ is \'{e}tale and that its relevant orbit has presentation growth
$O\bigl(n^2\log(n+2)\bigr)=o(d^n)$,
because $T|_S$ is affine-linear and $d\geq2$.
Proposition~\ref{prop:blowup} also gives
$\widetilde F^*\cI_E\subseteq\cI_E^{d}$ and $\widetilde F^\ell(C_{\widetilde F})\subseteq E$.
All hypotheses of Theorem~\ref{thm:line-bundle} are therefore satisfied for
$(\widetilde x_0,\widetilde V)$.  Equivariance gives
$\widetilde F^n(\widetilde x_0)\in\widetilde V \quad\Longleftrightarrow\quad F^n(x_0)\in V$
for every $n\geq0$.  Hence the return set for $(x_0,V)$ equals a return set
for $\widetilde F$ and is a finite union of arithmetic progressions.
We restore the first $N$ iterates to complete the proof.
\end{proof}

\subsection{Nonlinear bases and the degree gap}

For nonlinear \'{e}tale bases, we compare the degree of an iterate
of the base map with the radial order of the same iterate. The strict
inequality in \eqref{eq:degree-gap} makes the presentation-growth bound
\eqref{eq:cocycle-growth} of Lemma~\ref{lem:cocycle} smaller than the
radial contraction scale.

\begin{corollary}[\'{E}tale polynomial bases with a degree gap]
\label{cor:degree-gap}
Let $m,r\geq1$, and let
$T:\A^m_k\to\A^m_k$ be an \'{e}tale polynomial endomorphism. Let
$M(u)\in\operatorname{GL}_r(k[u_1,\ldots,u_m])$.
Fix $d\geq2$ and
\[
 0\ne\rho(u,\boldsymbol z)\in
 (z_1,\ldots,z_r)^{d-1}
 k[u_1,\ldots,u_m,z_1,\ldots,z_r],
\]
and define
$F(u,\boldsymbol z)= \bigl(T(u),\rho(u,\boldsymbol z)M(u)\boldsymbol z\bigr)$.
Let $Z=\A^m\times\{0\}$ and let $C$ be the vertical critical locus of
$F$.  Suppose that there are $q,\ell\geq1$ such that
\begin{equation}\label{eq:degree-gap}
 \deg(T^q)<d^q
 \qquad\text{and}\qquad
 F^\ell(C)\subseteq Z
\end{equation}
set-theoretically.  Then $F$ satisfies the Dynamical Mordell--Lang
conjecture.
\end{corollary}

\begin{proof}
We put $G=F^q$ and first show, by induction on $n$, that the fiber
coordinate of $F^n$ has the form
\begin{equation}\label{eq:iterate-radial-form}
 H_n(u,\boldsymbol z)=
 \rho_n(u,\boldsymbol z)M_n(u)\boldsymbol z,
 \qquad
 M_n(u)=M(T^{n-1}u)\cdots M(u),
\end{equation}
where $M_n$ is polynomially invertible and
$\rho_n\in(z_1,\ldots,z_r)^{d^n-1}$.
Indeed, $\rho_1=\rho$, and composition gives
\begin{equation}\label{eq:radial-recurrence}
 \rho_{n+1}(u,\boldsymbol z)
 =\rho_n(u,\boldsymbol z)\rho(T^nu,H_n(u,\boldsymbol z)).
\end{equation}
Since $H_n$ has fiber order at least $d^n$, the right-hand side of \eqref{eq:radial-recurrence} has
order at least $(d^n-1)+(d-1)d^n=d^{n+1}-1$.  The matrix $M_n$ is
polynomially invertible as a product of polynomially invertible matrices.

If $\rho_q=0$, then $F^q$ maps the whole space into $Z$, and the conclusion
follows by applying Theorem~\ref{thm:etale-dml} to the orbit tail in $Z$.
We therefore assume that $\rho_q\ne0$.  Formula
\eqref{eq:iterate-radial-form} shows that $G$ has the radial form of
Proposition~\ref{prop:blowup}, with radial order $d^q$ and base map $T^q$.

We let $C_G$ be the vertical critical locus of $G$.  The vertical chain rule
gives
\[
 \det D_{\boldsymbol z}H_q(u,\boldsymbol z)
 =\prod_{i=0}^{q-1}
 \det D_{\boldsymbol z}H
 \bigl(F^i(u,\boldsymbol z)\bigr),
\]
where $H=H_1$.  Thus
$C_G=\bigcup_{i=0}^{q-1}F^{-i}(C)$ set-theoretically. If
$x\in F^{-i}(C)$, then $F^i(x)\in C$ and
$F^{i+\ell}(x)\in Z$. Since $Z$ is invariant, $F^{qa}(x)\in Z$ whenever
$qa\geq i+\ell$. Hence
$G^a(C_G)\subseteq Z$ for every integer $a$ satisfying
$qa\geq q-1+\ell$.

To prove Dynamical Mordell--Lang for $G$, we fix a point $x$ and a closed
subvariety $V$, lift $x$ to the blow-up of $Z$, and replace $V$ by its
inverse image.  Proposition~\ref{prop:blowup} lifts $G$ to this blow-up.  The
base of the resulting line bundle is $\A^m\times\Pj^{r-1}$, on which the
induced map is the projective-linear cocycle associated with $T^q$ and
$M_q$.  By Lemma~\ref{lem:cocycle}, every orbit of this base map has
presentation growth
\begin{equation}\label{eq:iterated-cocycle-growth}
 \begin{cases}
 O(n^2\log(n+2)),&\deg(T^q)=1,\\
 O\bigl(n\deg(T^q)^n\bigr),&\deg(T^q)\geq2.
 \end{cases}
\end{equation}
The first bound in \eqref{eq:iterated-cocycle-growth} is $o(d^{qn})$.
For the second bound,
$\frac{n\deg(T^q)^n}{d^{qn}}
=n\left(\frac{\deg(T^q)}{d^q}\right)^n\longrightarrow0$
by the strict degree inequality in \eqref{eq:degree-gap}. Hence the
required growth is $o(d^{qn})$
after the orbit and target have been descended to a finitely generated
field.  Proposition~\ref{prop:blowup} also supplies radial contraction of
order $d^q$ and critical absorption.  Hence
Theorem~\ref{thm:line-bundle} applies to the lifted orbit and target.  The
equivariant blow-down argument used in the proof of Theorem~\ref{thm:main}
shows that the return set for $(x,V)$ is a finite union of arithmetic
progressions.  Since $(x,V)$ was arbitrary, $G$ satisfies Dynamical
Mordell--Lang.

Finally, we note that for $0\leq b<q$ and $j\geq0$,
$F^{b+qj}(x)=G^j(F^b(x))$.
The return set for $F$ is therefore the union, over $0\leq b<q$,
of the images of these return sets for $G$ under $j\mapsto b+qj$.
Each image is a finite union of arithmetic progressions, proving the
result for $F$.
\end{proof}

\subsection{Homogeneous radial factors}

For a homogeneous radial factor, we use the determinant formula
\eqref{eq:intro-determinant} to identify the critical locus explicitly and
remove the need for a separate critical-absorption hypothesis.

\begin{corollary}[Homogeneous radial factors]\label{cor:homogeneous}
In Theorem~\ref{thm:main}, suppose that
$P(u,\boldsymbol z)$ is a nonzero polynomial homogeneous of degree $d-1$
in $\boldsymbol z$.  Then every map
$F(u,\boldsymbol z)= \bigl(T(u),P(u,\boldsymbol z)M(u)\boldsymbol z\bigr)$
satisfies the Dynamical Mordell--Lang conjecture.  No separate
critical-absorption hypothesis is required.  The same conclusion holds in
the setting of Corollary~\ref{cor:degree-gap}: for an \'{e}tale polynomial
base, the degree-gap hypothesis alone is sufficient.
\end{corollary}

\begin{proof}
By \eqref{eq:intro-determinant}, the vertical critical locus is
$V(P)$ set-theoretically.  Its image lies in $Z$, so we apply
Theorem~\ref{thm:main}, or Corollary~\ref{cor:degree-gap}, with $\ell=1$.
\end{proof}

\begin{remark}[What is needed beyond a common radial factor]
\label{rem:scope}
The common scalar factor ensures that the lifted projective direction is
independent of the radial coordinate on the entire line bundle, not merely
on its zero section.  This distinction is essential.  For a general
polynomial vector map, the induced map on the exceptional divisor
alone does not provide a skew-product map on the whole line bundle.

The same transfer argument applies to another system if its blow-up is
the total space of a line bundle $\pi:L\to B$ and the original map lifts
to an endomorphism $\Phi:L\to L$ for which there is an \'{e}tale map
$\varphi:B\to B$ satisfying $\pi\circ\Phi=\varphi\circ\pi$ on all
of $L$.  We must also verify the remaining transfer hypotheses: for some
$d\geq2$ the zero-section ideal satisfies
$\Phi^*\cI_{Z_L}\subseteq\cI_{Z_L}^{d}$; one fixed iterate sends the
relative critical locus of $\Phi$ into $Z_L$; and the relevant base orbit
has presentation growth $o(d^n)$.  Under these explicit hypotheses,
Theorem~\ref{thm:line-bundle} applies and equivariance transfers its
conclusion through the blow-down.  None of these conditions, including
the existence of the skew-product lift itself, is automatic for an
arbitrary polynomial vector map.
\end{remark}

\section{Skolem--Mahler--Lech applications}
\label{sec:sml}

The Skolem--Mahler--Lech theorem describes the zero set of a
constant-coefficient linear recurrence as a finite union of arithmetic
progressions, including singletons. Ghioca--Xie obtained the same conclusion
for recurrences whose coefficients were rational functions evaluated along
an orbit of a rational map of degree at least two, wherever those evaluations
were defined \cite[Corollary~1.4]{GhiocaXieSkewLinear}. Wibmer treated
polynomial coefficients evaluated at the integers under the assumption that
the constant trailing coefficient was nonzero
\cite[Corollary~3.3]{Wibmer2015}. Here the \emph{trailing coefficient} is the
coefficient $h_0$ of the earliest term in \eqref{eq:scalar-recurrence}.
Our application concerns joint relations between such linear sequences and
the nonlinear orbits of the preceding sections.

\begin{definition}[The sequence ring and the shift]\label{def:sequence-ring}
We work over $\mathbb C$. Write $\operatorname{Seq}_{\mathbb C}$
for the ring of complex sequences modulo \emph{eventual equality}:
two sequences represent the same element if they agree at all
sufficiently large indices. Operations are componentwise, and the
shift is $\sigma((a_n)_{n\geq0})=(a_{n+1})_{n\geq0}$.
The support of a sequence is the set of indices at which it is
nonzero, regarded modulo finite changes. A sequence is
\emph{eventually nonzero} when every sufficiently late term is
nonzero; this is stronger than saying that it is a nonzero element
of $\operatorname{Seq}_{\mathbb C}$. We identify a complex number
with its constant sequence.
\end{definition}

\begin{definition}[Total quotient rings]\label{def:total-quotient-ring}
For a commutative ring $R$, its \emph{total quotient ring}
$\operatorname{Quot}(R)$ is obtained by inverting all non-zero-divisors,
so its elements are fractions with denominators that annihilate
no nonzero element of $R$. This need not be a single field when
$R$ has zero divisors.
\end{definition}
\begin{definition}[Idempotents and shift-fixed elements]\label{def:idempotents}
An \emph{idempotent} is an element $e$ satisfying $e^2=e$; nonzero
idempotents are \emph{primitive} if they cannot be written as a sum
of two nonzero idempotents whose product is zero. Idempotents whose
product is zero are called \emph{orthogonal}. The notation $L^\sigma$
denotes the elements fixed by the shift.
\end{definition}

In the theorem below, $p$
denotes an eventual period, not the residue characteristic used in
the earlier sections. Likewise, $\sigma$ now denotes the shift,
not the polynomial size function \eqref{eq:polynomial-size}.
A \emph{state} is the full tuple $\boldsymbol x_n$ in
\eqref{eq:augmented-state};
a polynomial in $h+1$ states may use every coordinate of each of
$\boldsymbol x_n,\ldots,\boldsymbol x_{n+h}$. The word
\emph{uniform} refers to the common period after the nonlinear map,
its initial point, and the matrix recurrence have been fixed. This
period does not depend on the polynomial relations or on the number
of successive states. The finite exceptional set may depend on both;
no common exceptional initial segment is asserted.
The sequence ring, total quotient ring, and idempotents in
Theorem~\ref{thm:uniform-sml} are those of
Definitions~\ref{def:sequence-ring},~\ref{def:total-quotient-ring},
and~\ref{def:idempotents}.

\begin{theorem}[Uniform Skolem--Mahler--Lech]
\label{thm:uniform-sml}
Let $F(u,\boldsymbol z)
=\bigl(T(u),\rho(u,\boldsymbol z)M(u)\boldsymbol z\bigr)$
satisfy the hypotheses of Theorem~\ref{thm:main} or
Corollary~\ref{cor:degree-gap}, and fix an orbit
$(u_n,\boldsymbol z_n)=F^n(u_0,\boldsymbol z_0)$.
Choose $a\geq1$ and $A(t)\in\operatorname{GL}_a(\mathbb C[t])$, and set
\begin{equation}\label{eq:augmented-state}
 Y_0=I_a,\qquad Y_{n+1}=A(n)Y_n,\qquad
 \boldsymbol x_n=(n,u_n,\boldsymbol z_n,Y_n),
\end{equation}
where the entries of $Y_n$ are regarded as separate coordinates.

There is an integer $p\geq1$, depending only on $F$, the initial point,
and $A$, with the following property. For every $h\geq0$ and every pair
of polynomials $P,D$ in the coordinates of $h+1$ states, the set
\begin{equation}\label{eq:joint-return-set}
 \left\{n\geq0:
 \begin{array}{l}
 P(\boldsymbol x_n,\ldots,\boldsymbol x_{n+h})=0,\\
 D(\boldsymbol x_n,\ldots,\boldsymbol x_{n+h})\ne0
 \end{array}
 \right\}
\end{equation}
agrees, outside a finite set, with a union of residue classes modulo
$p$. The exceptional finite set may depend on $h,P,D$, but $p$ does not.

Moreover, let $\mathcal R\subseteq\operatorname{Seq}_{\mathbb C}$ be
the $\mathbb C$-algebra generated by the coordinate sequences of
$\boldsymbol x_n$. Its total quotient ring embeds in
$\operatorname{Seq}_{\mathbb C}$ and has a decomposition
$L:=\operatorname{Quot}(\mathcal R) \simeq L_0\times\cdots\times L_{p-1}$
into fields. Shift extends injectively to $L$, permutes its primitive
idempotents in a single cycle, and satisfies $L^\sigma=\mathbb C$.
\end{theorem}

\begin{proof}
The augmented sequence is an orbit of the polynomial map
\[
 G(t,u,\boldsymbol z,Y)
 =\bigl(t+1,T(u),
        \rho(u,\boldsymbol z)M(u)\boldsymbol z,A(t)Y\bigr).
\]
Its base map is
$B(t,u,Y)=(t+1,T(u),A(t)Y)$.
When $T$ is affine-linear, we discard a finite initial segment and
restrict to its stabilized image, where $T$ is an automorphism.
The map $B$ is then \'{e}tale in both cases under consideration.

We put $D_A=\max_{i,j}\deg A_{ij}$, with zero entries omitted.
The matrix multiplying $Y$ in $B^q$ is
$A(t+q-1)\cdots A(t)$, so
$\deg(B^q)\leq \max\{\deg(T^q),\,1+qD_A\}$.
We choose $q_0$ with $\deg(T^{q_0})<d^{q_0}$; for an affine
automorphism, $q_0=1$ suffices. Since
$\deg(T^{kq_0})\leq\deg(T^{q_0})^k$ for positive integers $k$,
sufficiently large multiples
$q=kq_0$ satisfy $\deg(B^q)<d^q$.
The radial order and critical-absorption condition are unchanged by
the augmentation. Thus Corollary~\ref{cor:degree-gap} applies to $G$.
If $\rho$ vanishes identically after restriction to the stabilized
image, the fiber instead becomes zero after one step, and
Theorem~\ref{thm:etale-dml} applies to the remaining base orbit.
Consequently every element of $\mathcal R$ has a zero set that is
a finite union of arithmetic progressions.

We obtain the common period from the sequence-ring argument of
Wibmer \cite[proof of Theorem~3.1]{Wibmer2015} and his appendix to
Ghioca--Xie \cite[Appendix, proof of Theorem~4.2]{GhiocaXieSkewLinear};
the argument is included for completeness.
The ring $\mathcal R$ is reduced, finitely generated over $\mathbb C$,
and stable under the injective shift.
Suppose, for a contradiction, that a non-zero-divisor $b\in\mathcal R$ has infinitely many
zeros. An arithmetic progression of zeros would give
$b\,\sigma(b)\cdots\sigma^{q-1}(b)=0$ for some $q\geq1$; we choose $q$
minimal. Since $b$ is a non-zero-divisor,
$\sigma(b)\cdots\sigma^{q-1}(b)=0$. This product equals
$\sigma\bigl(b\sigma(b)\cdots\sigma^{q-2}(b)\bigr)$, so injectivity of
$\sigma$ contradicts the minimality of $q$. Therefore every
non-zero-divisor of $\mathcal R$ is eventually nonzero, and its termwise
reciprocal defines an element of $\operatorname{Seq}_{\mathbb C}$.

The universal property of localization therefore gives an injective
homomorphism $L=\operatorname{Quot}(\mathcal R)\hookrightarrow
\operatorname{Seq}_{\mathbb C}$. Since $\mathcal R$ is reduced and
Noetherian, its total quotient ring is a finite product of fields.
The shift of an eventually nonzero sequence is eventually nonzero,
so shift preserves non-zero-divisors and extends injectively to $L$.
If there are $p$ primitive idempotents, their images are
$p$ nonzero orthogonal idempotents summing to $1$, so shift permutes
them. The permutation consists of a single cycle. Otherwise, the sum of the
primitive idempotents in a proper union of cycles would be a nontrivial
shift-fixed idempotent. A shift-fixed sequence is eventually constant, and
an eventually constant idempotent is $0$ or $1$; hence no such nontrivial
idempotent exists. The same observation gives $L^\sigma=\mathbb C$.

Each primitive idempotent is fixed by $\sigma^p$.
After deleting finitely many indices, their supports partition the
integers into $p$ nonempty periodic sets, so each support is one
residue class modulo $p$.
On such a class, an element of $L$ is either eventually zero or
eventually nonzero, according as its component in the corresponding
field vanishes or is invertible.
Finally, every shifted coordinate sequence belongs to $\mathcal R$.
We apply this observation to the evaluated numerator and denominator
to obtain the asserted common period for the return set
\eqref{eq:joint-return-set}.
\end{proof}

To express the conclusion for a scalar sequence, we consider the recurrence
\begin{equation}\label{eq:scalar-recurrence}
 b_{n+a}=\sum_{i=0}^{a-1}h_i(n)b_{n+i},
 \qquad
 h_0\in\mathbb C^\times,\quad
 h_1,\ldots,h_{a-1}\in\mathbb C[t].
\end{equation}
Its companion matrix advances the state vector
$(b_n,\ldots,b_{n+a-1})^{\mathsf T}$ by one step: its first $a-1$
rows have a single $1$ just above the main diagonal and zero entries
elsewhere, and its last row is $(h_0(t),\ldots,h_{a-1}(t))$.
It belongs to $\operatorname{GL}_a(\mathbb C[t])$, since its
determinant is $(-1)^{a-1}h_0$. Here we form $Y_n$ using this
companion matrix in the recurrence of
Theorem~\ref{thm:uniform-sml}. Every solution state vector is $Y_n v$ for some
$v\in\mathbb C^a$. The period in Theorem~\ref{thm:uniform-sml}
therefore works for all solutions of this fixed recurrence and all
polynomial relations involving them and the nonlinear orbit.
For example, for every polynomial $H$ and every such solution,
$\{n\geq0:H(n,u_n,\boldsymbol z_n)=b_n\}$
is eventually a union of residue classes modulo the same $p$.
Finite simultaneous conditions and rational expressions are included
by combining polynomial equalities and denominator inequalities.

\section{Examples}
\label{sec:examples}

We present examples covering the one-dimensional specialization, homogeneous
higher-rank ramification, a nonlinear base satisfying the degree gap, and a
nonhomogeneous factor absorbed after two iterates. The notation $V(h)$
denotes the zero set of $h$. Multiplicities refer to exponents in the
critical determinant, and an unspecified $d$ is an arbitrary integer at
least two.

\begin{example}[One fiber coordinate]
We take $r=1$ and $M=1$. Writing $\rho(u,z)=z^{d-1}g(u,z)$ gives
\[
 F(u,z)=\bigl(T(u),z^d g(u,z)\bigr),\qquad
 C=V\!\left(\frac{\partial(z^d g)}{\partial z}\right).
\]
Theorem~\ref{thm:main}
specializes to the one-dimensional fiber family under the same
critical-absorption hypothesis.
\end{example}

\begin{example}[A radial map in every rank]\label{ex:basic-radial}
We take $r\geq2$, a nonzero linear form $L(\boldsymbol z)$
(a homogeneous polynomial of degree one in the fiber coordinates), and
$B\in\operatorname{GL}_r(k)$.  The map
$F(\boldsymbol z)=L(\boldsymbol z)^{d-1}B\boldsymbol z$ on $\A^r$
satisfies Dynamical Mordell--Lang by Corollary~\ref{cor:homogeneous}.
Its vertical critical locus is the hyperplane $L=0$, with multiplicity $r(d-1)$ in the critical
determinant. The entire hyperplane is mapped to the origin in one
step.
\end{example}

\begin{example}[A two-dimensional fiber]
We take $r=2$, $d=2$, $P(x,y)=x$, and $M=I_2$, obtaining
$F(x,y)=(x^2,xy)$.
By \eqref{eq:intro-determinant}, the critical line $x=0$ has
multiplicity two and is contracted to $(0,0)$.
Corollary~\ref{cor:homogeneous} proves Dynamical Mordell--Lang for this map,
which is neither \'{e}tale nor coordinatewise split.
\end{example}

\begin{example}[A nonconstant matrix cocycle]\label{ex:matrix-cocycle}
We take $T(u)=u+1$ and set
$M(u)=\begin{psmallmatrix}1&u\\0&1\end{psmallmatrix}$ and $P(u,x,y)=x+uy$.
Then
$F(u,x,y)= \bigl(u+1,(x+uy)^2,(x+uy)y\bigr)$.
The two fiber coordinates are coupled both through $P$ and through $M(u)$.
Formula \eqref{eq:intro-determinant} identifies the critical hypersurface
as $x+uy=0$, with multiplicity two.  It maps to the zero section in one
step, so Corollary~\ref{cor:homogeneous} gives Dynamical Mordell--Lang
on $\A^3$.
\end{example}

\begin{example}[A nonlinear \'{e}tale base]\label{ex:nonlinear-base}
We fix $a\in k^{\times}$ and consider the H\'{e}non
automorphism
$T(x,y)=(y,y^2-ax)$.
It has degree $2$ and constant Jacobian determinant $a$.  We take
$M(x,y)=\begin{psmallmatrix}1&x\\0&1\end{psmallmatrix}$ and
$P(x,y,z_1,z_2)=z_1^2+yz_1z_2+z_2^2$,
and define
$F(x,y,\boldsymbol z)= \bigl(T(x,y),P(x,y,\boldsymbol z)M(x,y)\boldsymbol z\bigr)$.
Here the radial order is $d=3$, while $\deg T=2<3$.  The polynomial $P$ is
homogeneous of fiber degree two, so its critical hypersurface is absorbed by
the zero section in one step. Corollaries~\ref{cor:degree-gap}
and~\ref{cor:homogeneous} therefore give Dynamical Mordell--Lang on $\A^4$.
This example has both a nonlinear base and a nonconstant matrix cocycle.
\end{example}

\begin{example}[A nonhomogeneous radial factor]\label{ex:nonhomogeneous-radial}
We take $T(u)=2u$ and
$\rho(u,z)=\frac{27}{8}uz(1-uz)$.
For $r=1$ the resulting map is
$F(u,z)=\left(2u,\frac{27}{8}uz^2(1-uz)\right)$.
Its critical locus is
$V(u)\cup V(z)\cup V(3uz-2)$. The first two components are mapped to
the zero section in one step.  If $w=uz$, then
$w\mapsto 27w^2(1-w)/4$.  The remaining critical component has $w=2/3$,
whose successive $w$-coordinates are
$\frac23\longmapsto1\longmapsto0$.
Hence the critical locus is absorbed after two iterates, and
Theorem~\ref{thm:main} applies even though $\rho$ is not homogeneous.
\end{example}

\bibliographystyle{amsalpha}
\bibliography{DML_Complex_Skew_Linear}

\end{document}